\documentclass[12pt,a4paper,reqno]{amsart}  
\usepackage{amsmath, amssymb, amsthm} 
\usepackage[english]{babel} 
\usepackage{fullpage}  
\usepackage{enumerate}
\usepackage[all]{xy}
\usepackage{graphicx}
\usepackage{mathtools}
\usepackage{mathrsfs} 
\usepackage{hyperref}
\usepackage{verbatim}
\numberwithin{equation}{section}
\usepackage{tikz}
\usepackage[pagewise]{lineno}

\address[]{\textit{Xiao Han}
\newline \indent 
Queen Mary University of London,
}
\email{x.h.han@qmul.ac.uk}

\def\ot{\otimes}

\def\stac#1{\raise-.2cm\hbox{$\stackrel{\,\displaystyle\otimes\,}{\scriptscriptstyle{#1}}$}}

\newtheorem{thm}{Theorem}[section]

\theoremstyle{definition}

\newtheorem{rem}[thm]{Remark}

\newcommand{\rosso}[1]{{\color{red}{#1}}}

\newcommand{\id}{\mathrm{id}}

\newcommand{\Aut}{\mathrm{Aut}}

\usepackage{xcolor}
\usepackage{microtype}

\numberwithin{equation}{section}

\theoremstyle{plain}
\newtheorem{theorem}{Theorem}[section]
\newtheorem{proposition}[theorem]{Proposition}
\newtheorem{lemma}[theorem]{Lemma}

\theoremstyle{definition}
\newtheorem{definition}[theorem]{Definition}

\theoremstyle{remark}
\newtheorem{remark}[theorem]{Remark}

\newcommand{\beq}{\begin{equation}}
\newcommand{\eeq}{\end{equation}}

\newcommand{\one}[1]{{#1}{}_{\scriptscriptstyle{(1)}}}
\newcommand{\two}[1]{{#1}{}_{\scriptscriptstyle{(2)}}}
\newcommand{\three}[1]{{#1}{}_{\scriptscriptstyle{(3)}}}
\newcommand{\four}[1]{{#1}{}_{\scriptscriptstyle{(4)}}}
\newcommand{\five}[1]{{#1}{}_{\scriptscriptstyle{(5)}}}
\newcommand{\six}[1]{{#1}{}_{\scriptscriptstyle{(6)}}}
\newcommand{\seven}[1]{{#1}{}_{\scriptscriptstyle{(7)}}}

\newcommand{\nul}[1]{{#1}{}^{\scriptscriptstyle{(0)}}}
\newcommand{\eins}[1]{{#1}{}^{\scriptscriptstyle{(-1)}}}

\newcommand{\uone}[1]{{#1}^{\scriptscriptstyle{(1)}}}
\newcommand{\utwo}[1]{{#1}^{\scriptscriptstyle{(2)}}}
\newcommand{\uthree}[1]{{#1}^{\scriptscriptstyle{(3)}}}

\newcommand{\yi}[1]{{#1}{}^{\tilde{1}}}
\newcommand{\er}[1]{{#1}{}^{\tilde{2}}}
\newcommand{\san}[1]{{#1}{}^{\tilde{3}}}

\renewcommand{\o}{{}_{\scriptscriptstyle{(1)}}}

\renewcommand{\t}{{}_{\scriptscriptstyle{(2)}}}

\renewcommand{\th}{{}_{\scriptscriptstyle{(3)}}}

\newcommand{\Ad}{\mathrm{Ad}}

\begin{document}

\keywords{Hopf algebroid, quantum group, Cayley algebra, coherent 2-group.}
\title{On coherent Hopf 2-Algebras}
\author{Xiao Han}

\begin{abstract}
    We construct a coherent Hopf 2-algebra in terms of Hopf coquasigroups, which relax the coassociativity condition and generalize the results in \cite{XH2023}. We also study quasi coassociative Hopf coquasigroups, and show that they give rise to coherent Hopf 2-algebras with nontrivial coassociators. As an example, we investigate the algebra of functions on a Cayley algebra basis.
\end{abstract}

\maketitle
\begin{quote}  2010 Mathematics Subject Classification: 16T05, 17A35, 20G42\end{quote}
\tableofcontents

\section{Introduction}

 In this paper, we study coherent Hopf 2-algebras (or noncoassociative quantum 2-groups) as a generalization of strict Hopf 2-algebras \cite{XH2023}.  Since Hopf coquasigroups \cite{Majid09} are noncoassociative quantum groups that provide many interesting examples, we are motivated to construct coherent Hopf 2-algebras from Hopf coquasigroups. Moreover, the coassociators of Hopf coquasigroups will play an important role in the coherence conditions of  coherent Hopf 2-algebras.

Recall that for a classical coherent 2-group, all the 1-arrows and 2-arrows are weakly invertible. Moreover, for the set of 2-arrows carries two products: the `horizontal' and the `vertical' products, which form a nonassociative group and a groupoid, respectively. Therefore, by applying the idea of `2-arrow' quantization,  a coherent Hopf 2-algebra could consist of two Hopf coquasigroups, which correspond to the `quantum 1-arrows' and `quantum 2-arrows'. Moreover, for the `quantum 2-arrows', it is on the one hand,  a Hopf coquasigroup corresponds to the `horizontal' coproduct; on the other hand, a Hopf algebroid corresponds to the `vertical' coproduct. These two coproducts also satisfy the interchange law. As a coherent Hopf 2-algebra, the coherence condition is described by a coassociator, which satisfies the `3-cocycle' condition.

We also study crossed comodules of Hopf coquasigroups as a generalization of crossed comodules of Hopf algebras \cite{Yael11}. We show that if a crossed comodule of a Hopf coquasigroup is quasi coassociative, one can construct a coherent Hopf 2-algebra with a nontrivial coassociator. As an example, we study the Hopf coquasigroup consisting of functions on a Cayley algebra basis.  We show that this Hopf coquasigroup is quasi coassociative and its coassociator is controlled by a 3-coboundary.

The paper is organized as follows: In \S 2 and \S 3, we will provide a brief introduction to coherent 2-groups, Hopf coquasigroups and Hopf algebroids. In \S 4, we will define coherent Hopf 2-algebras and study their properties. In \S 5, we will introduce crossed comodules of Hopf coquasigroups and quasi coassociative Hopf coquasigroups, which are shown to be coherent Hopf 2-algebras  under some conditions.  In \S 6, we study a special  coherent Hopf 2-algebra, namely, a Hopf coquasigroup which consists of functions on a Cayley algebra basis.

\section{Coherent 2-groups}\label{sec. coherent 2-group}
In this section, we will introduce quasigroups and coherent 2-groups  \cite{BL} \cite{Wockel09}.

\subsection{Quasigroups}
By \cite{Majid09}, we have the definition of quasigroup: 
\begin{definition}\label{def. quasigroup}
A \textit{quasigroup} is a set $G$ with a product and identity, such that for each element $g$ there is an inverse $g^{-1}\in G$, which satisfies $g^{-1}(gh)=h$ and $(hg^{-1})g=h$ for any $h\in G$.  
\end{definition}
For a quasigroup, the \textit{multiplicative associator} $\beta: G^{3}\to G$ is defined by
\begin{align}
    g(hk)=\beta(g, h, k)(gh)k,
\end{align}
for any $g, h, k\in G$. The group of associative elements $N(G)$ is given by
$$N(G)=\{a\in G | (ag)h=a(gh),\quad g(ah)=(ga)h,\quad (gh)a=g(ha),\quad \forall g, h\in G\},$$
which sometimes called the `nucleus'. A quasigroup is called \textit{quasiassociative}, if $\beta$ has its image in $N(G)$ and $uN(G)u^{-1}\subseteq N(G)$ for any $u\in G$. 
It is clear that any element in $N(G)$ can `pass through' the brackets of a product. For example, $(g(hx))k=(gh)(xk)$  for any $g, h, k\in G$ and $x\in N(G)$. By \cite{Majid09}, we also have the following Lemma:

\begin{lemma}\label{lemma. 3-cocycle}
Let $G$ be a quasiassociative quasigroup, then we have the following 3-cocycle condition:
\begin{align}\label{3cocycle}
   (g\beta(h, k, l)g^{-1})\beta(g, hk, l)\beta(g, h, k)=\beta(g, h, kl) \beta(gh, k, l)
\end{align}
for any $ g, h, k, l\in G$.
\end{lemma}
Lemma \ref{lemma. 3-cocycle} will be useful in the proof of Theorem \ref{thm. quasigroup constrcuct a coherent 2-group}.

\subsection{Coherent 2-groups}

By \cite{Wockel09}, we know  that a coherent 2-group is a monoidal category, in which every object is weakly invertible and every morphism is invertible. More precisely, we have the following definition: 
\begin{definition}\label{def. coherent 2-group}
A \textit{coherent 2-group} is a monoidal category $(G, \otimes, I, \alpha, r, l)$, with the unit $I$ and three natural isomorphisms, namely,  the associator, with components $\alpha_{g,h,k}: (g\otimes h)\ot k\to g\ot (h\ot k)$,  the right and left unitor,  with components $r_{g}: g\otimes I\to g$ and $l_{g}: I\otimes g\to g$,  such that  the following diagrams commute:
\begin{itemize}
    \item[(1)] The pentagon diagram:  {\small
 \[
  \xymatrix@=10pt{ &&{\scriptstyle (g\otimes h)\otimes (k\otimes l)}
  \ar[drr]^(.5){\displaystyle\alpha_{g,h,k\otimes l}}\\
  {\scriptstyle ((g\otimes h)\otimes k)\otimes
  l}\ar[urr]^(.5){\displaystyle\alpha_{g\otimes
  h,k,l}}\ar[ddr]_{\displaystyle\alpha_{g,h,k}\otimes \id_{l}}&&&&
  {\scriptstyle g\otimes(h\otimes (k\otimes l))}\\
  \\
  &{\scriptstyle (g\otimes (h\otimes k))\otimes
  l}\ar[rr]_{\displaystyle\alpha_{g,h\otimes k,l}}&&{\scriptstyle
  g\otimes((h\otimes k)\otimes l)}\ar[uur]_{\displaystyle\id_{g}\otimes
  \alpha_{h,k,l}} }
 \]
 }
 \item[(2)]  {\small
 \begin{equation*}
  \xymatrix{
  {\scriptstyle (g\otimes I)\otimes h}\ar[dr]_(.4){\displaystyle r_{g}\otimes\id_{h}}\ar[rr]^{\displaystyle \alpha_{g,I,h}}&& {\scriptstyle g\otimes (I \otimes h)}\ar[dl]^(.4){\displaystyle \id_{g}\otimes l_{h}}\\
  &  {\scriptstyle g\otimes h}}
 \end{equation*}
 }
 \item[(3)]
 Moreover, there are an additional functor $\iota: G\to G$, and  two natural equivalences, with components $i_{g}: g\otimes \iota(g)\to I$ and $e_{g}: \iota(g)\otimes g\to I$, such that the following diagram commutes:
 {\small
 \begin{equation*}
  \xymatrix{
  {\scriptstyle(g\otimes \iota(g))\otimes g}\ar[d]_{\displaystyle \alpha_{g,\iota(g),g}}\ar[rr]^(.55){\displaystyle i_{g}\otimes\id_{g}}&&{\scriptstyle I\otimes g}\ar[r]^(.6){\displaystyle l_{g}}&\ar[d]^{\displaystyle \id_{g}}{\scriptstyle g}\\
  {\scriptstyle g\otimes (\iota(g)\otimes g)}\ar[rr]^(.55){\displaystyle \id_{g}\otimes e_{g}}&&{\scriptstyle g\otimes I}\ar[r]^(.6){\displaystyle r_{g}}&{\scriptstyle g}
  }
 \end{equation*}
 }
\end{itemize}
A \textit{strict 2-group} is a coherent 2-group, such that the natural transformations $\alpha$, $l$, $r$, $i$ and $e$ are identities.
\end{definition}

In general, the objects of a coherent 2-group can be any unital set with a binary operation. However, in this paper, we are interested in a more restricted case where the objects of the corresponding monoidal category form a quasigroup, such that $l, r, i, e$ are identity natural transformations. Moreover, for a quasiassociative quasigroup we have

\begin{theorem}\label{thm. quasigroup constrcuct a coherent 2-group}
If $G$ is a quasiassociative quasigroup then $N(G)\ltimes G$ is a coherent 2-group.
\end{theorem}
\begin{proof}

To have a coherent 2-group, we first need to construct the set of morphisms by $H:=N(G)\ltimes G$, we can see $H$ is a quasigroup with the product given by:
\begin{align*}
    (n, g)\otimes (m, h):=(n(gmg^{-1}), gh).
\end{align*}
The inverse corresponding to this product is given by 
\begin{align*}
    (n, g)^{-1}:=(g^{-1}n^{-1}g, g^{-1}),
\end{align*}
which is well defined since $G$ is quasiassociative. Before we check if $H$ is a quasigroup, we can see that the adjoint action $\Ad: G\to \Aut(N(G))$ given by $\Ad_{g}(m):=gmg^{-1}$ is well defined for any $g\in G$ and $m\in N(G)$, since for any $g\in G$ and $m,n\in N(G)$, we have $(gmg^{-1})(gng^{-1})=((gm)g^{-1})(g(ng^{-1}))=(gm)(g^{-1}(g(ng^{-1})))=(gm)(ng^{-1})=g(mng^{-1})=g(mn)g^{-1}$. For any $g,h\in G$ and $m\in N(G)$, we have
$((gh)m(gh)^{-1})g=(g(hmh^{-1})g^{-1})g$. Indeed, on the one hand 
\begin{align*}
    ((gh)m(gh)^{-1})g=((gh)m(gh)^{-1})((gh)h^{-1})=(((gh)m(gh)^{-1})(gh))h^{-1}=((gh)m)h^{-1},
\end{align*}
on the other hand
\begin{align*}
    (g(hmh^{-1})g^{-1})g=&g(hmh^{-1})=((gh)h^{-1})(hmh^{-1})=(gh)(h^{-1}(hmh^{-1}))\\
    =&(gh)(mh^{-1})=((gh)m)h^{-1}.
\end{align*}
Therefore, by multipling $g^{-1}$ on the right side of $((gh)m(gh)^{-1})g=(g(hmh^{-1})g^{-1})g$,  we have $(gh)m(gh)^{-1}=g(hmh^{-1})g^{-1}$. So
\begin{align*}
    \Ad_{gh}(m)=(gh)m(gh)^{-1}=g(hmh^{-1})g^{-1}=\Ad_{g}\circ \Ad_{h} (m).
\end{align*}
Now we can see $H$ is a quasigroup:
\begin{align*}
    &((n, g)\otimes (m, h))\otimes (m, h)^{-1}=(n(gmg^{-1}), gh)\otimes (h^{-1}m^{-1}h, h^{-1})\\
    =&((n(gmg^{-1}))(\Ad_{gh}(\Ad_{h^{-1}}(m^{-1}))), (gh)h^{-1})=((n(gmg^{-1}))(gm^{-1}g^{-1}), g)\\
    =&(n, g),
\end{align*}
for any $(n, g)$, $(m, h)\in H$. Similarly, we have $(n, g)^{-1}\otimes ((n, g)\otimes (m, h))=(m, h)$.

Then we can construct the set of objects by $G$ and the tensor product $\otimes: G\times G\to G$ is defined to be the multiplication of $G$. The source and target maps $s, t: H\to G$ are given by
\begin{align*}
    s(n, g):=g, \qquad t(n, g):=ng.
\end{align*}
The identity morphism $\id: G\to H$ is given by
\begin{align*}
    \id(g):=(1, g).
\end{align*}
The composition $\circ: H{}_{s}\times_{t}H\to H$ is given by 
\begin{align*}
   (n, mg)\circ(m, g):=(nm, g). 
\end{align*}
The composition inverse is given by 
\begin{align*}
    (n, g)^{*}:=(n^{-1}, ng).
\end{align*}
Thus we get a groupoid with the associative composition, inverse, source and target maps. Moreover, the maps $s$, $t$ and $\id$ preserve the tensor product. Indeed, by using the fact that $G$ is quasiassociative we have
\begin{align*}
    &t((n, g)\otimes (m, h))=t(n(gmg^{-1}), gh)=(n(gmg^{-1}))gh=n((gmg^{-1})(gh))\\
    &=n(((gmg^{-1})g)h)
    =n((gm)h)=n(g(mh))=(ng)(mh)=t(n, g)\otimes t(m, h).
\end{align*}
The interchange law of products can be checked as follows:
On the one hand, we have

\begin{align*}
    ((n_{1}, m_{1}g)&\circ(m_{1},g))\ot((n_{2},m_{2}h)\circ(m_{2},h))=(n_{1}m_{1}, g)\ot(n_{2}m_{2},h)\\
    =&(n_{1}m_{1} \Ad_{g}(n_{2}m_{2}),gh).
\end{align*}
On the other hand, 
\begin{align*}
  ((n_{1}, m_{1}g)&\ot(n_{2},m_{2}h))\circ ((m_{1},g)\ot(m_{2},h))\\
  =&(n_{1}\Ad_{m_{1}g}(n_{2}), (m_{1}g)(m_{2}h))\circ (m_{1}\Ad_{g}(m_{2}), gh)\\
  =&(n_{1}\Ad_{m_{1}g}(n_{2})m_{1}\Ad_{g}(m_{2}),gh)\\
  =&(n_{1}m_{1}\Ad_{g}(n_{2})m_{1}^{-1}m_{1}\Ad_{g}(m_{2}),gh)\\
   =&(n_{1}m_{1}\Ad_{g}(n_{2})\Ad_{g}(m_{2}),gh)\\
   =&(n_{1}m_{1}\Ad_{g}(n_{2}m_{2}),gh)
\end{align*}
where the second step is well defined since
\begin{align*}
   m_{1}\Ad_{g}(m_{2})(gh)=&m_{1}((gm_{2}g^{-1})(gh))=m_{1}(((gm_{2}g^{-1})g)h)=m_{1}((gm_{2})h)=m_{1}(g(m_{2}h))\\
   =&(m_{1}g)(m_{2}h).
\end{align*}

The associator $\alpha$ is given by
\begin{align*}
    \alpha_{g, h, k}:=(\beta(g, h, k), (gh)k),
\end{align*}
this is well defined since the image of $\beta$ belongs to $N(G)$, and the source of $\alpha_{g, h, k}$ is $(gh)k$ and the target of $\alpha_{g, h, k}$ is $g(hk)$. Because $G$ is a quasigroup, we have $\alpha_{g, g^{-1}, h}=\alpha_{h, g^{-1}, g}=(1, h)=\id_{h}$ and $\alpha_{1,g,h}=\alpha_{g,1,h}=\alpha_{g,h,1}=\id_{gh}$ for any $g, h\in G$.
We can also see that $\alpha$ is a natural transformation. Indeed, if 
$(l, g)$, $(m, h)$,$(n, k)\in H$, we can see on the one hand
\begin{align*}
    \alpha_{lg, mh, nk}&\circ (((l, g)\otimes (m, h))\otimes (n, k))\\
    =&(\beta(lg, mh, nk), ((lg)(mh))(nk))\circ ((l(\Ad_{g}(m)))(\Ad_{gh}(n)), (gh)k)\\
    =&(\beta(lg, mh, nk)\Big((l(\Ad_{g}(m)))(\Ad_{gh}(n))\Big), (gh)k),
\end{align*}
and on the other hand 
\begin{align*}
    ((l, g)\otimes& ((m, h)\otimes (n, k)))\circ \alpha_{g, h, k}\\
    =&(l \Ad_{g}\big(m \Ad_{h}(n)\big), g(hk))\circ (\beta(g, h, k), (gh)k)\\
    =&(\Big(l \Ad_{g}\big(m \Ad_{h}(n)\big)\Big)\beta(g, h, k), (gh)k).
\end{align*}
Since 
\begin{align*}
    \beta(lg, mh, nk)\Big((l(\Ad_{g}(m)))(\Ad_{gh}(n))\Big) (gh)k=& \beta(lg, mh, nk)(l(\Ad_{g}(m)))\Big((\Ad_{gh}(n)) (gh)k\Big)\\
    =&\beta(lg, mh, nk)(l(\Ad_{g}(m)))\Big((gh)(nk)\Big)\\
    =&\beta(lg, mh, nk)l\Big((\Ad_{g}(m))(gh)\Big)(nk)\\
    =&\beta(lg, mh, nk)l\Big(g(mh)\Big)(nk)\\
    =&\beta(lg, mh, nk)\Big((lg)(mh)\Big)(nk)\\
    =&(lg)((mh)(nk))\\
    =&\Big(l \Ad_{g}\big(m \Ad_{h}(n)\big)\Big)\beta(g, h, k) (gh)k,
\end{align*}
we get
\begin{align*}
    \beta(lg, mh, nk)\Big((l(\Ad_{g}(m)))(\Ad_{gh}(n))\Big)=\Big(l \Ad_{g}\big(m \Ad_{h}(n)\big)\Big)\beta(g, h, k),
\end{align*}
therefore,
\begin{align*}
    \alpha_{lg, mh, nk}\circ (((l, g)\otimes (m, h))\otimes (n, k))=((l, g)\otimes ((m, h)\otimes (n, k)))\circ \alpha_{g, h, k}.
\end{align*}
The pentagon diagram can be proved by Lemma \ref{lemma. 3-cocycle}. Indeed, let $g, h, k, l\in G$,  on the one hand we have
\begin{align*}
  &\alpha_{g, h, kl}  \circ \alpha_{g h, k, l}
  =(\beta(g, h, k l), (gh)(kl))\circ (\beta(gh, k, l), ((gh)k)l)\\
  =&(\beta(g, h, k l)\beta(gh, k, l), ((gh)k)l)
\end{align*}
On the other hand,
\begin{align*}
    &(\id_{g}\otimes \alpha_{h, k, l})\circ \alpha_{g, hk, l}\circ (\alpha_{g, h, k}\otimes id_{l})\\
    =&(g\beta(h, k, l)g^{-1}, g((hk)l))\circ \big(\beta(g, hk, l), (g(hk))l\big)\circ (\beta(g, h, k), ((gh)k)l)\\
    =&((g\beta(h, k, l)g^{-1})\beta(g, hk, l)\beta(g, h, k), ((gh)k)l),
\end{align*}
by using Lemma \ref{lemma. 3-cocycle} we get the pentagon.
For 
the natural transformations $l, r, e, i$, we can see $l_{g}=r_{g}=(1,g)$, and $i_{g}=e_{g}=(1,1)$ for any $g\in G$, which satisfy all the axioms of a coherent 2-group. 
\end{proof}

\section{Hopf coquasigroups and Hopf algebroids}

In this section, we will recall some material about Hopf algebras \cite{Majid}, Hopf coquasigroups and their modules and comodules.
We will also review Hopf algebroids over commutative rings.

\begin{definition}\label{def. Hopf algebra}
A \textit{bialgebra} is an algebra $H$ with two algebra maps $\Delta_{H}: H\to H\ot H$ (called the coproduct) and $\epsilon_{H}: H\to \mathbb{C}$ (called the counit), such that
\begin{align}\label{bialgebra}
    (\id_{H}\ot \Delta_{H})\circ \Delta_{H}=(\Delta_{H}\ot \id_{H})\circ \Delta,\quad 
(\id_{H}\otimes \epsilon_{H}) \circ \Delta_{H} =\id_{H}=(\epsilon_{H}\otimes \id_{H}) \circ \Delta_{H},    
\end{align}
where the algebra multiplication on $H\ot H$ is given by $(h\ot g)\cdot (h'\ot g'):=(hh'\ot gg')$ for any $h\ot g, h'\ot g'\in H\ot H$. Moreover, if there is a linear map $S: H\to H\ot H$ (called the antipode), such that
\begin{align}\label{antipode}
  m_{H}\circ (S\ot \id_{H})\circ \Delta_{H}=1_{H}\epsilon_{H}=  m_{H}\circ (\id_{H}\ot S)\circ \Delta_{H},
\end{align}
where $m_{H}$ is the product on $H$, then $H$ is called a \textit{Hopf algebra}.

\end{definition}

For the coproduct of a bialgebra, we use the sumless Sweedler notation
$\Delta_{H}(h)=\one{h}\ot\two{h}$,
and
its iterations: $\Delta^n=(\id_{H} \ot  \Delta_H) \circ\Delta_H^{n-1}: h \mapsto \one{h}\ot \two{h} \ot
\cdots \ot h_{\scriptscriptstyle{(n+1)\;}}$. In order to define a quantization of a quasigroup, we will relax the condition that the coproduct is coassociative. In other words, the coproduct of a bialgebra will no longer satisfy the first equation of (\ref{bialgebra}). By \cite{Majid09}, we have a quantization of a quasigroup:
\begin{definition}
A {\em Hopf coquasigroup} $H$ is an unital associative algebra, equiped with counital algebra homomorphisms $\Delta: H\to H\otimes H$, $\epsilon: H\to \mathbb{C}$ (called the coproduct and counit), and a linear map $S_{H}: H\to H$ (called the antipode) such that
\begin{equation}
    (m_{H}\otimes \id_{H})(S_{H}\otimes \id_{H}\otimes \id_{H})(\id_{H}\otimes \Delta)\Delta=1\otimes \id_{H}=(m_{H}\otimes \id_{H})(\id_{H}\otimes S_{H}\otimes \id_{H})(\id_{H}\otimes \Delta)\Delta,
\end{equation}
\begin{equation}
    (\id_{H}\otimes m_{H})(\id_{H}\otimes S_{H}\otimes \id_{H})(\Delta\otimes \id_{H})\Delta=\id_{H}\otimes 1= (\id_{H}\otimes m_{H})(\id_{H}\otimes \id_{H}\otimes S_{H})(\Delta\otimes \id_{H})\Delta.
\end{equation}
A morphism between two Hopf coquasigroups is an algebra map $f: H\to G$, such that for any $h\in H$, $\one{f(h)}\ot \two{f(h)}=f(\one{h})\ot f(\two{h})$ and $\epsilon_{G}(f(h))=\epsilon_{H}(h)$.
\end{definition}

\begin{remark}
Since  the coproduct of a Hopf coquasigroup is not necessarily coassociative, we cannot use the  Sweedler index notion for the iterated coproduct $\Delta^{n}$ (but we still use $\one{h}\otimes \two{h}$ as the image of the coproduct $\Delta$). For example, since $(\Delta\ot\id)\circ\Delta\neq (\id\ot \Delta)\circ \Delta$ we won't have: $\one{\one{h}}\otimes \two{\one{h}}\otimes \two{h}= \one{h}\otimes\two{h}\otimes \three{h}=\one{h}\otimes \one{\two{h}}\otimes \two{\two{h}}$.  It is given in \cite{Majid09} that the antipode $S_{H}$ of a Hopf coquasigroup also satisfies: 
\begin{itemize}
    \item $\one{h}S_{H}(\two{h})=\epsilon(h)=S_{H}(\one{h})\two{h}$,
    \item $S_{H}(hh')=S_{H}(h')S_{H}(h)$,
    \item $\one{S_{H}(h)}\otimes \two{S_{H}(h)}=S_{H}(\two{h})\otimes S_{H}(\one{h})$,
\end{itemize}
for any $h, h'\in H$.
\end{remark}

Given a Hopf coquasigroup $H$,
we can define a linear map $\beta: H\to H\otimes H\otimes H$ (called the comultiplicative coassociator) by 
\begin{align}\label{coassociator}
    \beta(h):=\one{\one{h}}\one{\one{S_{H}(\two{h})}}\otimes \one{\two{\one{h}}}\two{\one{S_{H}(\two{h})}}\otimes \two{\two{\one{h}}}\two{S_{H}(\two{h})}
\end{align}
 for any $h\in H$.
We can see that 
\begin{align}\label{equ. relation of coassociator}
\beta \ast ((\Delta\otimes \id_{H})\circ \Delta)=(\id_{H}\otimes \Delta)\circ \Delta, 
\end{align}
where $\ast$ is the convolution product in the vector space
$H':=\mathrm{Hom}(H,H\ot H\ot H)$, namely $(f \ast g) (h):=f(\one{h})g(\two{h})$ for any $f,g\in H'$. More precisely, (\ref{equ. relation of coassociator}) can be written as 
\begin{align*}
    \one{\one{\one{h}}}&\one{\one{S_{H}(\two{\one{h}})}}\one{\one{\two{h}}}\otimes \one{\two{\one{\one{h}}}}\two{\one{S_{H}(\two{\one{h}})}}\two{\one{\two{h}}}\otimes \two{\two{\one{\one{h}}}}\two{S_{H}(\two{\one{h}})}\two{\two{h}}\\
    =&\one{h}\otimes \one{\two{h}}\otimes \two{\two{h}}.
\end{align*}
 For any $h\in H$, we will always denote the image of $\beta$ by (with sumless index) $$\beta(h)=h^{\hat{1}}\otimes h^{\hat{2}}\otimes h^{\hat{3}}.$$  A Hopf coquasigroup $H$ is a Hopf algebra if and only if $\beta(h)=\epsilon(h)1_{H}\otimes 1_{H}\otimes 1_{H}$.

Given a Hopf coquasigroup $H$, a \textit{left $H$-comodule} is a vector space $V$ carrying a
left $H$-coaction, namely a linear map $\delta^{V} : V\to H\otimes V$
such that
\begin{align}
(\id_{H}\ot  \delta^{V})\circ \delta^{V} & = (\Delta\ot  \id_{V})\circ \delta^{V} ~,  \quad
(\epsilon\ot  \id_{V}) \circ \delta^{V} =\id_{V} ~. \label{eqn:lHcomodule}
\end{align}
In the sumless Sweedler notation, $\delta^{V}: v\mapsto \eins{v}\ot  \nul{v}$, and the left $H$-comodule properties
read
\begin{align*}
\one{\eins{v}}\otimes \two{\eins{v}}\otimes \nul{v} = \eins{v}\otimes \eins{\nul{v}}\otimes \nul{\nul{v}} ~,\quad 
\epsilon (\eins{v}) \,\nul{v}= v ~,
\end{align*}
for all $v\in V$.

In particular, a \textit{left $H$-comodule algebra} is an algebra  $A$,  such that the coaction $\delta: A\to H\ot  A$ is an algebra map. A \textit{left $H$-comodule coalgebra} is a coalgebra $C$, which is a left $H$-comodule and such that the
coproduct and the counit of $C$ are morphisms of $H$-comodules. Explicitly, this means that, for each $c \in C$,
\begin{align*}
\eins{c}\ot \one{\nul{c}}\ot \two{\nul{c}}&=\eins{\one{c}}\eins{\two{c}}\ot \nul{\one{c}}\ot \nul{\two{c}} \, ,
\\
\epsilon_{C}(c)&=\eins{c}\epsilon_{C}(\nul{c}) \, .
\end{align*}

\begin{definition}\label{def. coassocaitive pair}
A \textit{coassociative pair} $(A, B, \phi)$ consists of a Hopf coquasigroup $B$ and a Hopf algebra $A$, together with a Hopf coquasigroup morphism $\phi: B\to A$, such that
\begin{align}\label{definition. coassocaitive pair}
  \begin{cases}
  \phi(\one{\one{b}})\otimes \two{\one{b}}\otimes \two{b}=&\phi({\one{b}})\otimes \one{\two{b}}\otimes \two{\two{b}}\\
    \one{\one{b}}\otimes \phi(\two{\one{b}})\otimes \two{b}=&{\one{b}}\otimes \phi(\one{\two{b}})\otimes \two{\two{b}}\\
      \one{\one{b}}\otimes \two{\one{b}}\otimes \phi(\two{b})=&\one{b}\otimes \one{\two{b}}\otimes \phi(\two{\two{b}}).
  \end{cases}
\end{align}
\end{definition}

\begin{remark}
A coassociative pair can be viewed as the dual case of a group and a quasigroup, such that there is a quasigroup morphism which maps the group into the associative elements of the quasigroup. More precisely, let $H$ be a group, $G$ be a quasigroup, and $\phi: H\to G$ be a morphism of quasigroup, such that the $\phi(H)\subseteq N(G)$. Then we have \begin{align*}
      \begin{cases}
(\phi(h)g)g'=\phi(h)(gg')\\
    (g\phi(h))g'=g(\phi(h)g')\\
g(g'\phi(h))=(gg')\phi(h),
  \end{cases}
\end{align*} 
for any $h\in H$ and $g,g'\in G$. 
\end{remark}

For a Hopf coquasigroup, the $n$-th iterated coproducts $\Delta^{n}_{I}$ (for $n\geq 2$)  are not always equal, where we use index I to distinguish different kinds of iterated product. We have:

\begin{lemma}\label{lemma. coquasigroup}
Let $(A, B, \phi)$ be a coassociative pair of a Hopf algebra $A$ and a Hopf coquasigroup $B$, and let $\Delta_{I}^{n}$, $\Delta_{J}^{n}$ be n-th iterated coproducts of $B$ with $\Delta_{I}^{n}(b)=b_{I_{1}}\otimes b_{I_{2}}\otimes\cdots\otimes b_{I_{n+1}}$ and $\Delta_{J}^{n}(b)=b_{J_{1}}\otimes b_{J_{2}}\otimes\cdots\otimes b_{J_{n+1}}$. If 
\begin{align*}
    b_{I_{1}}\otimes& b_{I_{2}}\otimes\cdots\otimes\epsilon_{B}(b_{I_{m_{1}}})\otimes\cdots\otimes\epsilon_{B}(b_{I_{m_{k}}})\otimes\cdots\otimes b_{I_{n+1}}\\
    =&b_{J_{1}}\otimes b_{J_{2}}\otimes\cdots\otimes\epsilon_{B}(b_{J_{m_{1}}})\otimes\cdots\otimes\epsilon_{B}(b_{J_{m_{k}}})\otimes\cdots\otimes b_{J_{n+1}}
\end{align*}
for $1 \leq m_{1} < m_{2}< \cdots<m_{k}\leq n+1$ (and $m_{i+1}-m_{i}\leq 1$ for any $1\leq i\leq k-1$), then
\begin{align*}
    b_{I_{1}}\otimes &b_{I_{2}}\otimes\cdots\otimes\phi(b_{I_{m_{1}}})\otimes\cdots\otimes\phi(b_{I_{m_{k}}})\otimes\cdots\otimes b_{I_{n+1}}\\
    =&b_{J_{1}}\otimes b_{J_{2}}\otimes\cdots\otimes\phi(b_{J_{m_{1}}})\otimes\cdots\otimes\phi(b_{J_{m_{k}}})\otimes\cdots\otimes b_{J_{n+1}}.
\end{align*}
\end{lemma}

\begin{proof}
First, we show that if 
\begin{align}\label{equ. proposition equ}
    b_{I_{1}}\otimes b_{I_{2}}\otimes\cdots\otimes\epsilon_{B}(b_{I_{m}})\otimes\cdots\otimes b_{I_{n+1}}=b_{J_{1}}\otimes b_{J_{2}}\otimes\cdots\otimes\epsilon_{B}(b_{J_{m}})\otimes\cdots\otimes b_{J_{n+1}}
\end{align}
for $1 \leq m \leq n+1$, then 
\begin{align*}
    b_{I_{1}}\otimes b_{I_{2}}\otimes\cdots\otimes\phi(b_{I_{m}})\otimes\cdots\otimes b_{I_{n+1}}=b_{J_{1}}\otimes b_{J_{2}}\otimes\cdots\otimes\phi(b_{J_{m}})\otimes\cdots\otimes b_{J_{n+1}}.
\end{align*}

We can prove this inductively. For $n=2$, this is obvious by the definition of a coassociative pair. 
Now we consider the case for $n\geq 3$.  We can see both sides of equation (\ref{equ. proposition equ}) are equal to the image of an $(n-1)$-th iterated coproduct $\Delta_{K}^{n-1}$, which can be written as $\Delta_{K}^{n-1}=(\Delta_{K'}^{p}\otimes \Delta_{K''}^{q})\circ \Delta$ for some iterated coproducts $\Delta_{K'}^{p}$, $\Delta_{K''}^{q}$ with $p+q=n-2$. Assume this proposition is correct for $n=N-1$. We have two cases for the index of $I_{m}$ and $J_{m}$:

The first case is that the first index of $b_{I_{m}}$ and $b_{J_{m}}$ are the same (where the first index means the first Sweedler index on the left, for example, the first index of $\one{\two{\two{b}}}$ is 2). When  the first indices of $b_{I_{m}}$ and $b_{J_{m}}$ are 1. In this case, we can see $\Delta_{I}^{n}=(\Delta_{I_{1}}^{p+1}\otimes \Delta_{K''}^{q})\circ \Delta$ and $\Delta_{J}^{n}=(\Delta_{J_{1}}^{p+1}\otimes \Delta_{K''}^{q})\circ \Delta$, for some $(p+1)$-th iterated coproducts $\Delta_{I_{1}}^{p+1}$ and $\Delta_{J_{1}}^{p+1}$. By applying the hypotheses for $\Delta_{I_{1}}^{p+1}$ and $\Delta_{J_{1}}^{p+1}$, we get the result. When  the first indices of $b_{I_{m}}$ and $b_{J_{m}}$ are 2, the situation is similar.

The second case is that the first indices of $b_{I_{m}}$ and $b_{J_{m}}$ are different. Assume the first index of $b_{I_{m}}$ is 1 and $b_{J_{m}}$ is 2. In this case $m$ has to be equal to $p+2$, and $\Delta_{I}^{n}=(\Delta_{E}^{p+1}\otimes \Delta_{K''}^{q})\circ \Delta$ and $\Delta_{J}^{n}=(\Delta_{K'}^{p}\otimes \Delta_{F}^{q+1})\circ \Delta$ for some iterated $(p+1)$-th coproduct
$\Delta_{E}^{p+1}$ with $(id_{B}^{\otimes p}\otimes \epsilon_{B})\circ \Delta_{E}^{p+1}=\Delta_{K'}^{p}$ and iterated $(q+1)$-th coproduct $\Delta_{F}^{q+1}$ with $(\epsilon_{B}\otimes id_{B}^{\otimes q})\circ \Delta_{F}^{q+1}=\Delta_{K''}^{q}$.
Define $\Delta_{G}^{p+1}:=(\Delta_{K'}^{p}\otimes \id_{B})\circ \Delta$ and $\Delta_{H}^{q+1}:=(\id_{B}\otimes\Delta_{K''}^{q})\circ \Delta$ (notice that $\Delta_{E}^{p+1}$ is not necessarily equal to $\Delta_{G}^{p+1}$, and $\Delta_{F}^{q+1}$ is not necessarily equal to $\Delta_{H}^{q+1}$), then we can see 
\begin{align*}
     b_{I_{1}}\otimes &b_{I_{2}}\otimes\cdots\otimes\phi(b_{I_{m}})\otimes b_{I_{m+1}}\otimes \cdots\otimes b_{I_{n+1}}\\
     =&b_{\scriptscriptstyle{(1)}E_{1}}\otimes b_{\scriptscriptstyle{(1)}E_{2}}\otimes\cdots\otimes\phi(b_{\scriptscriptstyle{(1)}E_{p+2}})\otimes b_{\scriptscriptstyle{(2)}K''_{1}}\otimes\cdots\otimes b_{\scriptscriptstyle{(2)}K''_{q+1}}\\
     =&b_{\scriptscriptstyle{(1)}G_{1}}\otimes b_{\scriptscriptstyle{(1)}G_{2}}\otimes\cdots\otimes\phi(b_{\scriptscriptstyle{(1)}G_{p+2}})\otimes b_{\scriptscriptstyle{(2)}K''_{1}}\otimes\cdots\otimes b_{\scriptscriptstyle{(2)}K''_{q+1}}\\
     =&b_{\scriptscriptstyle{(1)}\scriptscriptstyle{(1)}K'_{1}}\otimes b_{\scriptscriptstyle{(1)}\scriptscriptstyle{(1)}K'_{2}}\otimes\cdots\otimes b_{\scriptscriptstyle{(1)}\scriptscriptstyle{(1)}K'_{p+1}}\otimes\phi(\two{\one{b}})\otimes b_{\scriptscriptstyle{(2)}K''_{1}}\otimes\cdots\otimes b_{\scriptscriptstyle{(2)}K''_{q+1}}\\
     =&b_{\scriptscriptstyle{(1)}K'_{1}}\otimes b_{\scriptscriptstyle{(1)}K'_{2}}\otimes\cdots\otimes b_{\scriptscriptstyle{(1)}K'_{p+1}}\otimes\phi(\one{\two{b}})\otimes b_{\scriptscriptstyle{(2)}\scriptscriptstyle{(2)}K''_{1}}\otimes\cdots\otimes b_{\scriptscriptstyle{(2)}\scriptscriptstyle{(2)}K''_{q+1}}\\
     =&b_{\scriptscriptstyle{(1)}K'_{1}}\otimes b_{\scriptscriptstyle{(1)}K'_{2}}\otimes\cdots\otimes b_{\scriptscriptstyle{(1)}K'_{p+1}}\otimes\phi(b_{\scriptscriptstyle{(2)}H_{1}})\otimes b_{\scriptscriptstyle{(2)}H_{2}}\otimes\cdots\otimes b_{\scriptscriptstyle{(2)}H_{q+2}}\\
     =&b_{\scriptscriptstyle{(1)}K'_{1}}\otimes b_{\scriptscriptstyle{(1)}K'_{2}}\otimes\cdots\otimes b_{\scriptscriptstyle{(1)}K'_{p+1}}\otimes\phi(b_{\scriptscriptstyle{(2)}F_{1}})\otimes b_{\scriptscriptstyle{(2)}F_{2}}\otimes\cdots\otimes b_{\scriptscriptstyle{(2)}F_{q+2}}\\
     =&b_{J_{1}}\otimes b_{J_{2}}\otimes\cdots\otimes\phi(b_{J_{m}})\otimes b_{J_{m+1}}\otimes\cdots\otimes b_{J_{n+1}},
\end{align*}
where $b_{\scriptscriptstyle{(1)}E_{1}}\otimes b_{\scriptscriptstyle{(1)}E_{2}}\otimes\cdots\otimes b_{\scriptscriptstyle{(1)}E_{p+2}}\otimes b_{\scriptscriptstyle{(2)}K''_{1}}\otimes\cdots\otimes b_{\scriptscriptstyle{(2)}K''_{q+1}}:=\Delta_{E}^{p+1}(\one{b})\ot\Delta_{K''}^{q}(\two{b})$ and similar for the rest. The 2nd and 6th steps use the hypotheses for $n\leq N-1$, and the 4th step uses the definition of a coassociate pair.

If $k=2$, we can also prove this inductively. Now we need to consider the first index of $b_{I_{m_{1}}}$, $b_{I_{m_{2}}}$, $b_{J_{m_{1}}}$ and $b_{J_{m_{2}}}$. There are several cases. If all of them are equal to 1 or 2, we are done by hypothesis. If the first index of $b_{I_{m_{1}}}$ and $b_{J_{m_{1}}}$  are 1,  and first index of $b_{I_{m_{2}}}$ and $b_{J_{m_{2}}}$ is 2, then we go back to the case for $k=1$ by considering only the terms with the first index equal to 1 or 2. If the first index of $b_{I_{m_{1}}}$, $b_{J_{m_{1}}}$  and $b_{I_{m_{2}}}$ are 1 and the first index of $b_{J_{m_{2}}}$ is 2, then $b_{I_{m_{2}}}$ is the last term whose first index are 1 and $b_{J_{m_{2}}}$ is the first term whose first index are 2. Therefore, we can use the same method above for the $k=1$ case, where the first index of $b_{I_{m}}$ is 1 and $b_{J_{m}}$ is 2. If the first index of $b_{I_{m_{1}}}$, $b_{I_{m_{2}}}$ are 1 and the first index of $b_{J_{m_{1}}}$, $b_{J_{m_{2}}}$ are 2, then $b_{I_{m_{1}}}$, $b_{I_{m_{2}}}$ are the last two terms whose first index are 1 and $b_{J_{m_{1}}}$, $b_{J_{m_{2}}}$ are the first two terms whose first index are 2, we can also use the same method as above, the only different is to move two terms from the `left-hand' side to the `right-hand' side, we will omit the detail here.

\end{proof}

There is a dual version of the Hopf coquasigroup \cite{Majid09}:
\begin{definition}
A  {\em Hopf quasigroup} $A$ is a coassociative coalgebra with a coproduct $\Delta: A\to A\otimes A$  and counit $\epsilon: A\to k$, together with a unital and possibly nonassociative algebra structure, such that the coproduct and counit are algebra maps. Moreover, there is a linear map (the antipode) $S_{A}: A\to A$ such that:
\begin{equation}
    m(\id_{A}\otimes m)(S_{A}\otimes \id_{A}\otimes \id_{A})(\Delta\otimes \id_{A})=\epsilon\otimes \id_{A}=m(\id_{A}\otimes m)(\id_{A}\otimes S_{A}\otimes \id_{A})(\Delta\otimes \id_{A})
\end{equation}
\begin{equation}
    m(m\otimes \id_{A})(\id_{A}\otimes S_{A}\otimes \id_{A})(\id_{A}\otimes \Delta)=\id_{A}\otimes \epsilon=m(m\otimes \id_{A})(\id_{A}\otimes \id_{A}\otimes S_{A})(\id_{A}\otimes\Delta).
\end{equation}
\end{definition}
A Hopf quasigroup is a Hopf algebra if and only if it is associative. In the following, we will introduce central Hopf algebroids, which is a Hopf algebroid \cite{Boehm} with images of the source and target maps belonging to the center.

\begin{definition}\label{def. central Hopf algebroids}
Let $B$ be a commutative algebra, a \textit{central Hopf algebroid over $B$} is an algebra $H$ with two algebra maps (called source and target map) $s:B\to H$ and $t:B\to H$. In addition, there is an antialgebra map (called Hopf algebroid antipode) $S: H\to H$, such that:
\begin{itemize}
    \item [(1)]The image of $s$ and $t$ belongs to the center of $H$.
    \item[(2)] $(H, \Delta, \epsilon)$ is a $B$-coring with the $B$-bimodule structure given by $b\triangleright h\triangleleft b'=s(b)ht(b')$, for any $h\in H$ and $b,b'\in B$. Namely, there are two $B$-bimodule maps $\Delta:H\to H\ot_{B} H$ and $\epsilon:H\to B$, such that
    \[(\Delta\ot_{B} \id)\circ\Delta=(\id\ot_{B}\Delta)\circ\Delta,\quad (\epsilon\ot_{B} \id)\circ\Delta=\id=(\id\ot_{B} \epsilon)\circ\Delta,\]
    where the balanced tensor product $\ot_{B}$ is induced by the $B$-bimodule structure of $H$, i.e. $g\ot_{B}s(b)h=t(b) g \ot_{B}h$ for any $g,h\in H$ and $b\in B$.
    \item[(3)] $\Delta$ and $\epsilon$ are algebra maps. 
    \item [(4)] For any $h\in H$ and $b, b'\in B$, 
    \begin{align}\label{equ. antipode of Hopf algebroid}        S(t(b)hs(b'))=t(b')S(h)s(b).
    \end{align}
    \item[(5)] $m\circ (S\otimes_{B} \id_{H})\circ \Delta=t\circ \epsilon$ and $m\circ (\id_{H}\otimes_{B} S)\circ \Delta=s\circ \epsilon$,
\end{itemize}
 In (5), we can see the product $m$ factors through the balanced tensor product $\ot_{B}$ because of (\ref{equ. antipode of Hopf algebroid}).
\end{definition}

Let $H$ be a central Hopf algebroid over $B$. We define the convolution product $\star: {}_{B}\mathrm{Hom}_{B}(H, B)\times {}_{B}\mathrm{Hom}_{B}(H, B)\to {}_{B}\mathrm{Hom}_{B}(H, B)$ by $(f\star g)(h):=f(\uone{h})g(\utwo{h})$ for any $f, g\in {}_{B}\mathrm{Hom}_{B}(H, B)$ and $h\in H$, where ${}_{B}\mathrm{Hom}_{B}(H, B)$ is the vector space of $B$-bimodule maps and we use 
the upper Sweeder index notation for the coproduct Hopf algebroids, i.e. $\Delta(h)=\uone{h}\ot \utwo{h}$. By  \cite{BW},  we know ${}_{B}\mathrm{Hom}_{B}(H, B)$ is an algebra with unit $\epsilon: H\to B$.

\section{Coherent Hopf-2-algebras}\label{sec. coherent Hopf 2-algebra}

In the following, we generalize the definition of Hopf 2-algebras in \cite{XH2023} by relaxing the coassociativity condition.

\begin{definition}\label{def. coherent Hopf-2-algebra}
A {\em coherent Hopf 2-algebra} consists of a commutative Hopf coquasigroup\\
$(B, m_{B}, 1_{B}, \Delta_{B}, \epsilon_{B}, S_{B})$ and a Hopf coquasigroup $(H, m, 1_{H}, \blacktriangle, \epsilon_{H}, S_{H})$, and a central Hopf algebroid $(H, m, 1_{H}, \Delta, \epsilon, S)$ over $B$. Moreover, there is an algebra map (called coassociator) $\alpha: H\to B\otimes B\otimes B$,  such that all the structures above satisfy the following axioms:
\begin{itemize}
    \item [(i)] The underlying algebra of the Hopf coquasigroup $(H, m, 1_{H}, \blacktriangle, \epsilon_{H}, S_{H})$ and the Hopf algebroid $(H, m, 1_{H}, \Delta, \epsilon, S)$ are the same.
    \item[(ii)] $\epsilon: H\to B$ and $s, t: B\to H$ are morphisms of Hopf coquasigroups.
    \item[(iii)] The two coproducts $\Delta$ and $\blacktriangle$ satisfies the following cocommutation relation:
   \begin{align}\label{equ. cocommutative relation}
        (\blacktriangle\otimes \blacktriangle)\circ \Delta=(\id_{H}\otimes \mathrm{flip}\otimes \id_{H})\circ (\Delta\otimes_{B}\Delta)\circ \blacktriangle,
    \end{align}
    where $\id_{H}\ot \mathrm{flip} \ot \id_{H}:H\ot H\ot_{B\ot B}H\ot H\to H\ot_{B} H\ot H\ot_{B} H$ is given by  $\id_{H}\ot \mathrm{flip} \ot \id_{H}: X\ot  Y\ot_{B\ot B} Z\ot  W\mapsto (X\ot_{B}  Z)\ot (Y\ot_{B}W)$.
\end{itemize}

\begin{itemize}
    \item [(iv)] For the coassociator, 
    \begin{align}
         \alpha\circ t=(\Delta_{B}\otimes \id_{B})\circ \Delta_{B}, \qquad \alpha\circ s=(\id_{B}\otimes \Delta_{B})\circ \Delta_{B}.
    \end{align}

    \item[(v)]  Let $\star$ denote the convolution product corresponding to the coproduct of the Hopf algebroid, we have
    \begin{align}\label{equ. co-natural transformation}
        ((s\otimes s\otimes s)\circ \alpha)\star ((\blacktriangle\otimes \id_{H})\circ \blacktriangle)=((\id_{H}\otimes \blacktriangle)\circ \blacktriangle)\star ((t\otimes t\otimes t)\circ \alpha)
    \end{align}
    \item[(vi)]
     The 3-cocycle condition:
            \begin{equation}\label{equ. coassociator 3}
                ((\epsilon\otimes \alpha)\circ \blacktriangle)\star ((\id_{B}\otimes \Delta_{B}\otimes \id_{B})\circ \alpha)\star ((\alpha\otimes \epsilon)\circ \blacktriangle)
                =((\id_{B}\otimes \id_{B}\otimes \Delta_{B})\circ \alpha)\star ((\Delta_{B}\otimes \id_{B}\otimes \id_{B})\circ \alpha).            
            \end{equation}
               
\end{itemize}
A coherent Hopf 2-algebra is a Hopf 2-algebra \cite{XH2023}, if $H$ and $B$ are coassociative ($H$ and $B$ are Hopf algebras), and $\alpha=(\epsilon\otimes \epsilon\otimes \epsilon)\circ(\blacktriangle\otimes \id_{H})\circ\blacktriangle$.
\end{definition}

\begin{remark}\label{remark for coherent Hopf 2-algebra}
In general, for every Hopf algebroid over an algebra $B$, the base algebra $B$ is not necessarily commutative. However, in order to give a correct definition of coherent Hopf 2-algebras, we need the maps $\epsilon, s, t$ to be  Hopf algebra maps, since only in this case condition (v) and (vi) make sense. As a result, we assume that the Hopf algebroid $H$ is a central Hopf algebroid, and the base algebra $B$ is a commutative algebra. In the following, we denote the image of $\alpha$ by $\alpha(h)=:\yi{h}\otimes \er{h}\otimes \san{h}$ for any $h\in H$, and use different sumless Sweedler notations for the coproducts of Hopf coquasigroup and Hopf algebroid structure on $H$, namely,  $\blacktriangle(h)=:\one{h}\otimes\two{h}$, $\Delta(h)=:\uone{h}\otimes_{B} \utwo{h}$.

\begin{itemize}
   \item [(1)] By \cite{XH2023}, condition (iii) is well defined and can be  translated into 
    \begin{align}\label{coco}
\uone{\one{h}}\ot_{B}\utwo{\one{h}}\ot \uone{\two{h}}\ot_{B}\utwo{\two{h}}=\one{\uone{h}}\ot_{B}\one{\utwo{h}}\ot\two{\uone{h}}\ot_{B}\two{\utwo{h}},
    \end{align}{}
    for any $h\in H$.
    \item[(2)] By using condition (iv) and the fact that $s, t$ are bialgebra maps, we can see (\ref{equ. co-natural transformation}) is well defined since $(s\otimes s\otimes s)\circ \alpha\circ t=(\blacktriangle\otimes \id_{H})\circ \blacktriangle\circ s$ and $(t\otimes t\otimes t)\circ \alpha\circ s=(\id_{H}\otimes \blacktriangle)\circ \blacktriangle\circ t$. 
    \item[(3)] The left hand side of (\ref{equ. coassociator 3}) is well defined since
    \begin{align*}
        ((\epsilon\otimes \alpha)\circ \blacktriangle)(t(b))=&\epsilon(t(\one{b}))\otimes \alpha(t(\two{b}))=\one{b}\otimes \one{\one{\two{b}}}\otimes \two{\one{\two{b}}}\otimes \two{\two{b}}\\
        =&(\id_{B}\otimes\Delta_{B}\otimes \id_{B})\circ \alpha(s(b)).
    \end{align*}
    and
    \begin{align*}
        (\id_{B}\otimes\Delta_{B}\otimes \id_{B})\circ \alpha(t(b))=&\one{\one{b}}\otimes \one{\two{\one{b}}}\otimes \two{\two{\one{b}}}\otimes \two{b}=\alpha(s(\one{b}))\otimes \epsilon(s(\two{b}))\\
        =&((\alpha\otimes \epsilon)\circ\blacktriangle)(s(b)).
    \end{align*}
        The right hand side of (\ref{equ. coassociator 3}) is also well defined, since
        \begin{align*}
            ((\id_{B}\otimes \id_{B}\otimes \Delta_{B})\circ \alpha)(t(b))=\one{\one{b}}\otimes \two{\one{b}}\otimes \one{\two{b}}\otimes \two{\two{b}}=((\Delta_{B}\otimes \id_{B}\otimes \id_{B})\circ \alpha)(s(b)).
        \end{align*}
\end{itemize}

\end{remark}

Now let's explain why Definition \ref{def. coherent Hopf-2-algebra} is a quantisation of a coherent 2-group, whose objects form a quasigroup. First, the morphisms and their composition form a groupoid, which corresponds to a Hopf algebroid. Second, the tensor products of objects and morphisms form two quasigroups, which corresponds to two Hopf coquasigroups. 

By the definition of monoidal category, we can see that axiom (ii) is natural, since the source and target maps from objects to morphisms preserve the tensor product, and the identity map from objects to morphisms also preserves the tensor product. The interchange law corresponds to condition (iii). The source and target of the morphism $\alpha_{g,h,k}$ is $(gh)k$ and $g(hk)$, which corresponds to condition (iv).  The naturality of $\alpha$ corresponds to (v). The pentagon diagram corresponds to condition (vi). Here, we will continue to refer to Definition \ref{def. coherent Hopf-2-algebra} as a coherent Hopf 2-algebra, even though it corresponds only to a special case of a ‘quantum’ coherent 2-group whose sets of objects and morphisms form quasigroups.

\begin{proposition}\label{Prof. Hopf 2-algebra}
Given a coherent Hopf 2-algebra as in Definition \ref{def. coherent Hopf-2-algebra}, the antipodes satisfy the following property:

\item[(i)]$\Delta\circ S_{H}=(S_{H}\ot_{B} S_{H})\circ \Delta$.
\item[(ii)] $S$ is a coalgebra map on $(H, \blacktriangle, \epsilon_{H})$. In other words, $\blacktriangle\circ S=(S\ot S)\circ \blacktriangle$ and $\epsilon_{H}\circ S=\epsilon_{H}$.
\item[(iii)] If $H$ is commutative, $S\circ S_{H}=S_{H}\circ S$.

\end{proposition}{}
\begin{proof}

For (i),  $\forall h\in H$ we can see $S_{H}(\uone{h})\otimes_{B}S_{H}(\utwo{h})$ is well defined. Since the image of source and target maps belongs to the center of $H$, we have
\begin{align*}
    S_{H}(t(b)\uone{h})\otimes_{B}S_{H}(\utwo{h})=&S_{H}(t(b)) S_{H}(\uone{h})\otimes_{B}S_{H}(\utwo{h})
    =t(S_{B}(b)) S_{H}(\uone{h})\otimes_{B}S_{H}(\utwo{h})\\
    =& S_{H}(\uone{h})\otimes_{B} s(S_{B}(b))S_{H}(\utwo{h})=S_{H}(\uone{h})\otimes_{B}S_{H}(s(b)\utwo{h}).
\end{align*}
We can see
\begin{align*}
    S_{H}(\uone{h})&\otimes_{B}S_{H}(\utwo{h})\\
    =&(S_{H}(\uone{\one{\one{h}}})\otimes_{B}S_{H}(\utwo{\one{\one{h}}}))(\Delta(\two{\one{h}}S_{H}(\two{h})))\\
    =&(S_{H}(\uone{\one{\one{h}}})\otimes_{B}S_{H}(\utwo{\one{\one{h}}}))(\uone{\two{\one{h}}}\otimes_{B}\utwo{\two{\one{h}}})(\uone{(S_{H}(\two{h}))}\otimes_{B}\utwo{(S_{H}(\two{h}))})\\
    =&(S_{H}(\one{\uone{\one{h}}})\otimes_{B}S_{H}(\one{\utwo{\one{h}}}))(\two{\uone{\one{h}}}\otimes_{B}\two{\utwo{\one{h}}})(\uone{(S_{H}(\two{h}))}\otimes_{B}\utwo{(S_{H}(\two{h}))})\\
    =&(\epsilon_{H}(\uone{\one{h}})\otimes_{B}\epsilon_{H}(\utwo{\one{h}}))(\uone{(S_{H}(\two{h}))}\otimes_{B}\utwo{(S_{H}(\two{h}))})\\
    =&\epsilon_{B}(\epsilon(\uone{\one{h}})\epsilon(\utwo{\one{h}}))(\uone{(S_{H}(\two{h}))}\otimes_{B}\utwo{(S_{H}(\two{h}))})\\
    =&\epsilon_{B}(\epsilon(s(\epsilon(\uone{\one{h}}))\utwo{\one{h}}))(\uone{(S_{H}(\two{h}))}\otimes_{B}\utwo{(S_{H}(\two{h}))})\\
    =&\epsilon_{H}(\one{h})(\uone{(S_{H}(\two{h}))}\otimes_{B}\utwo{(S_{H}(\two{h}))})\\
    =&\uone{(S_{H}(h))}\otimes_{B}\utwo{(S_{H}(h))}.
\end{align*}
For (ii), we can first observe that for any $h\in H$, $(S(\one{\uone{h}})\otimes S(\two{\uone{h}}))(\one{\utwo{h}}\otimes\two{\utwo{h}})$ is well defined. Indeed, 
\begin{align*}
    (S(\one{(t(b)\uone{h})})&\otimes S(\two{(t(b)\uone{h})}))(\one{\utwo{h}}\otimes\two{\utwo{h}})\\
    =&(S(t(b\one{})\one{\uone{h}})\otimes S(t(b\two{})\two{\uone{h}}))(\one{\utwo{h}}\otimes\two{\utwo{h}})\\
    =&(S(\one{\uone{h}})\otimes S(\two{\uone{h}}))(s(b\one{})\ot s(b\two{}))(\one{\utwo{h}}\otimes\two{\utwo{h}})\\
    =&(S(\one{\uone{h}})\otimes S(\two{\uone{h}}))(s(b)\one{}\ot s(b)\two{})(\one{\utwo{h}}\otimes\two{\utwo{h}})\\
    =& (S(\one{\uone{h}})\otimes S(\two{\uone{h}}))(\one{(s(b)\utwo{h})}\otimes\two{(s(b)\utwo{h})}).
\end{align*}
Similarly, since $H$ is a central Hopf algebroid, $(\one{\uone{h}}\otimes\two{\uone{h}})(\one{S(\utwo{h})}\otimes \two{S(\utwo{h})})$ is well defined.
So on the one hand
\begin{align*}
    (S(\one{\uone{h}})&\otimes S(\two{\uone{h}}))(\one{\utwo{h}}\otimes\two{\utwo{h}})(\one{S(\uthree{h})}\otimes \two{S(\uthree{h})})\\
    =&(S(\one{\uone{h}})\otimes S(\two{\uone{h}}))(\one{\utwo{h}\uone{}}\otimes\two{\utwo{h}\uone{}})(\one{S(\utwo{h}\utwo{})}\otimes \two{S(\utwo{h}\utwo{})})\\
    =&(S(\one{\uone{h}})\otimes S(\two{\uone{h}}))(\blacktriangle(\utwo{h}\uone{}S(\utwo{h}\utwo{})))\\
    =&(S(\one{\uone{h}})\otimes S(\two{\uone{h}}))(\one{s(\epsilon(\utwo{h}))}\otimes \two{s(\epsilon(\utwo{h}))})\\
    =&(S(\one{\uone{h}})\otimes S(\two{\uone{h}}))(s(\epsilon(\one{\utwo{h}}))\otimes s(\epsilon(\two{\utwo{h}})))\\
    =&S(\one{\uone{h}}t(\epsilon(\one{\utwo{h}})))\otimes S(\two{\uone{h}}t(\epsilon(\two{\utwo{h}})))\\
    =&S(\uone{\one{h}}t(\epsilon(\utwo{\one{h}})))\otimes S(\uone{\two{h}}t(\epsilon(\utwo{\two{h}})))\\
    =&S(\one{h})\otimes S(\two{h}),
\end{align*}
on the other hand
\begin{align*}
     (S(\one{\uone{h}})&\otimes S(\two{\uone{h}}))(\one{\utwo{h}}\otimes\two{\utwo{h}})(\one{S(\uthree{h})}\otimes \two{S(\uthree{h})})\\
     =&(S(\one{\uone{h}\uone{}})\otimes S(\two{\uone{h}\uone{}}))(\one{\uone{h}\utwo{}}\otimes\two{\uone{h}\utwo{}})(\one{S(\utwo{h})}\otimes \two{S(\utwo{h})})\\
     =&(S(\one{\uone{h}}\uone{})\otimes S(\two{\uone{h}}\uone{}))(\one{\uone{h}}\utwo{}\otimes\two{\uone{h}}\utwo{})(\one{S(\utwo{h})}\otimes \two{S(\utwo{h})})\\
     =&t(\epsilon(\one{\uone{h}}))\otimes t(\epsilon(\two{\uone{h}}))(\one{S(\utwo{h})}\otimes \two{S(\utwo{h})})\\
     =&\one{(t(\epsilon(\uone{h}))S(\utwo{h}))}\otimes \two{(t(\epsilon(\uone{h}))S(\utwo{h}))}\\
     =&\blacktriangle(S(s(\epsilon(\uone{h}))\utwo{h}))\\
     =&\blacktriangle(S(h)),
\end{align*}
so $\blacktriangle(S(h))=S(h\one{})\ot S(h\two{})$. We also have \begin{align*}
    \epsilon_{H}(S(h))=&\epsilon_{H}(S(s\circ \epsilon(\utwo{h})\uone{h}))=\epsilon_{H}(S(\uone{h})t\circ\epsilon(\utwo{h}))=\epsilon_{H}(S(\uone{h}))\epsilon_{H}\circ t\circ\epsilon(\utwo{h})\\
    =&\epsilon_{H}(S(\uone{h}))\epsilon_{B}\circ\epsilon(\utwo{h})=\epsilon_{H}(S(\uone{h}))\epsilon_{H}(\utwo{h})=\epsilon_{H}(S(\uone{h})\utwo{h})=\epsilon_{H}(t(\epsilon(h)))\\
    =&\epsilon_{B}\circ \epsilon(h)=\epsilon_{H}(h).
\end{align*}
For (iii), we can first observe that 
 $S(S_{H}(\uone{h}))S_{H}(\utwo{h})$ is well defined. Indeed,
\begin{align*}
    S(S_{H}(t(b)\uone{h}))S_{H}(\utwo{h})=&S(S_{H}(\uone{h})t\circ S_{B}(b))S_{H}(\utwo{h})=S(S_{H}(\uone{h}))s\circ S_{B}(b)S_{H}(\utwo{h})\\
    =&S(S_{H}(\uone{h}))S_{H}(s(b)\utwo{h}).
\end{align*}
Similarly, $S_{H}(\uone{h})S_{H}(S(\utwo{h}))$ is also well defined. Indeed,
\begin{align*}
    S_{H}(t(b)\uone{h})S_{H}(S(\utwo{h}))=&S_{H}(\uone{h})S_{H}(t(b)))S_{H}(S(\utwo{h}))=S_{H}(\uone{h})S_{H}(t(b)S(\utwo{h}))\\
    =&S_{H}(\uone{h})S_{H}(S(s(b)\utwo{h})).
\end{align*}
So on the one hand
\begin{align*}
    &S(S_{H}(\uone{h}))S_{H}(\utwo{h})S_{H}(S(\uthree{h}))
    =S(S_{H}(\uone{h}))S_{H}(\utwo{h}S(\uthree{h}))\\
    =&S(S_{H}(\uone{h}))S_{H}(s(\epsilon(\utwo{h})))
    =S(S_{H}(\uone{h}))s(S_{B}(\epsilon(\utwo{h})))\\
    =&S\big(S_{H}(\uone{h})t(S_{B}(\epsilon(\utwo{h})))\big)
    =S\big(S_{H}(\uone{h})S_{H}(t(\epsilon(\utwo{h})))\big)\\
    =&S(S_{H}(\uone{h}t(\epsilon(\utwo{h}))))
    =S(S_{H}(h))
\end{align*}
where in the first step we use the fact that $H$ is commutative. On the other hand
\begin{align*}
     &S(S_{H}(\uone{h}))S_{H}(\utwo{h})S_{H}(S(\uthree{h}))
    =S(\uone{S_{H}(\uone{h})})\utwo{S_{H}(\uone{h})}S_{H}(S(\utwo{h}))\\
    =&t(\epsilon(S_{H}(\uone{h})))S_{H}(S(\utwo{h}))=S_{H}(t(\epsilon(\uone{h})))S_{H}(S(\utwo{h}))\\
    =&S_{H}(S(s(\epsilon(\uone{h}))\utwo{h}))=S_{H}(S(h)),
\end{align*}
where the first step uses (i) of this Proposition.

\end{proof}{}

\section{Crossed comodule of Hopf coquasigroups}\label{sec. Crossed comodule of Hopf coquasigroups}

In this section, we will study crossed comodules of Hopf coquasigroups. Moreover, we will use crossed comodules of Hopf coquasigroups to construct coherent Hopf 2-algebras with non-trivial coassociators.

\begin{definition}\label{cross comodule of Hopf coquasigroups}
A \textit{crossed comodule of Hopf coquasigroup} $(A, B, \phi,\delta)$ is a coassociative pair $(A, B, \phi)$,  such that
\begin{itemize}
    \item[(1)] $A$ is a left $B$ comodule coalgebra and left $B$ comodule algebra with coaction $\delta$;
    
    \item[(2)]For any $b\in B$,
    \begin{align}\label{cc3}
       \eins{\phi(b)}\ot \nul{\phi(b)}=\one{\one{b}}S_{B}(\two{b})\ot \phi(\two{\one{b}})=\one{b}S_{B}(\two{\two{b}})\otimes \phi(\one{\two{b}}); 
        \end{align}{}
        \item[(3)]For any $a\in A$,
     \begin{align}\label{cc4}
\phi(\eins{a})\ot \nul{a}=\one{a}S_{A}(\three{a})\ot \two{a}.
        \end{align}{}  
\end{itemize}{}

\end{definition}{}
If $B$ is coassociative
 the crossed comodule of Hopf coquasigroup is a \textit{crossed comodule of Hopf algebra} \cite{Yael11}. In the following, we will present a useful Lemma. We will still write down the proof as it is slightly different with Proposition 5.14 of \cite{Majid09}.

\begin{lemma}\label{lemma. construct Hopf coquasigroup}
Let $(A, B, \phi, \delta)$ be a crossed comodule of Hopf coquasigroup, if $B$ is commutative, then the tensor product $H:=A\ot B$  is a Hopf coquasigroup, with factorwise multiplication, and unit $1_{A}\ot 1_{B}$. The coproduct is given by $\blacktriangle(a\ot b):=\one{a}\ot \eins{\two{a}}\one{b}\ot \nul{\two{a}}\ot \two{b}$, the counit is given by $\epsilon_{H}(a\ot b):=\epsilon_{A}(a)\epsilon_{B}(b)$ and the antipode is given by $S_{H}(a\ot b):= S_{A}(\nul{a})\ot S_{B}(\eins{a}b)$,  $\forall a\ot b\in A\ot B$. Moreover, if $B$ is coassociative, then $H$ is a Hopf algebra.
\end{lemma}{}

\begin{proof}
$A\ot B$ is clearly an unital algebra and it is clearly counital. We first show $H$ is also a bialgebra:
\begin{align*}
    \blacktriangle(aa'\ot bb')=&\one{a}\one{a'}\ot \eins{\two{a}}\eins{\two{a'}}\one{b}\one{b'}\ot \nul{\two{a}}\nul{\two{a'}}\ot \two{b}\two{b'}\\
    =&\one{a}\one{a'}\ot \eins{\two{a}}\one{b}\eins{\two{a'}}\one{b'}\ot \nul{\two{a}}\nul{\two{a'}}\ot \two{b}\two{b'}\\
    =&\blacktriangle(a\ot b) \blacktriangle(a'\ot b'),
\end{align*}
here we use the fact that $B$ is a commutative algebra in the 2nd step.  Thus $H$ is a bialgebra. Now check the antipode $S_{H}$ on $h=a\otimes b$,
\begin{align*}
    \one{\one{h}}\otimes& S_{H}(\two{\one{h}})\two{h}\\
    =&\one{\one{a}}\otimes \eins{\two{\one{a}}}\one{\eins{\two{a}}}\one{\one{b}}\otimes S_{A}(\nul{\nul{\two{\one{a}}}})\nul{\two{a}}\otimes S_{B}(\eins{\nul{\two{\one{a}}}}\two{\eins{\two{a}}}\two{\one{b}})\two{b}\\
    =&\one{\one{a}}\otimes \one{\eins{\two{\one{a}}}}\one{\eins{\two{a}}}\one{\one{b}}\otimes S_{A}(\nul{\two{\one{a}}})\nul{\two{a}}\otimes S_{B}(\two{\eins{\two{\one{a}}}}\two{\eins{\two{a}}}\two{\one{b}})\two{b}\\
    =&\one{\one{a}}\otimes \one{\eins{S_{A}(\two{\one{a}})}}\one{\eins{\two{a}}}\one{\one{b}}\otimes \nul{S_{A}(\two{\one{a}})}\nul{\two{a}}\otimes S_{B}(\two{\eins{S_{A}(\two{\one{a}})}}\two{\eins{\two{a}}}\two{\one{b}})\two{b}\\
    =&a\ot \one{\one{b}}\ot1\ot S_{B}(\two{\one{b}})\two{b}\\
    =&a\otimes b\otimes 1_{A}\otimes 1_{B}
\end{align*}
where the third step uses the fact that $\eins{a}\otimes S_{A}(\nul{a})=\eins{S_{A}(a)}\otimes \nul{S_{A}(a)}$. Indeed,
\begin{align*}
    \eins{a}\otimes& S_{A}(\nul{a})\\
    =&\eins{\one{\one{a}}}\eins{\two{\one{a}}}\eins{S_{A}(\two{a})}\otimes S_{A}(\nul{\one{\one{a}}})\nul{\two{\one{a}}}\nul{S_{A}(\two{a})}\\
    =&\eins{\one{a}}\eins{S_{A}(\two{a})}\otimes S_{A}(\one{\nul{\one{a}}})\two{\nul{\one{a}}}\nul{S_{A}(\two{a})}\\
    =&\eins{\one{a}}\eins{S_{A}(\two{a})}\otimes \epsilon_{A}(\nul{\one{a}})\nul{S_{A}(\two{a})}\\
    =&\eins{S_{A}(a)}\otimes \nul{S_{A}(a)},
\end{align*}
where in the second and third steps we use the comodule coalgebra property.
The rest axioms of Hopf coquasigroups are similar. Thus $H$ is a Hopf coquasigroup.

When $B$ is coassociative, for any $a\ot b\in A\ot B$, we also have
\begin{align*}
((\id_{H}\ot \blacktriangle)\circ \blacktriangle)(a\ot b)=&(\id_{H}\ot \blacktriangle)(\one{a}\ot \eins{\two{a}}\one{b}\ot \nul{\two{a}}\ot \two{b})\\
=&\one{a}\ot \eins{\two{a}}\one{b}\ot \one{\nul{\two{a}}}\ot\eins{\two{\nul{\two{a}}}}\two{b}\ot \nul{\two{\nul{\two{a}}}}\ot \three{b}\\
=&\one{a}\ot \eins{\two{a}}\eins{\three{a}}\one{b}\ot \nul{\two{a}}\ot \eins{\nul{\three{a}}}\two{b}\ot \nul{\nul{\three{a}}}\ot \three{b}\\
=&\one{a}\ot \eins{\two{a}}\one{\eins{\three{a}}}\one{b}\ot \nul{\two{a}}\ot \two{\eins{\three{a}}}\two{b}\ot \nul{\three{a}}\ot \three{b}\\
=&(\blacktriangle\ot \id_{H})(\one{a}\ot \eins{\two{a}}\one{b}\ot \nul{\two{a}}\ot \two{b})\\
=&((\blacktriangle\ot \id_{H})\circ \blacktriangle)(a\ot b),
\end{align*}{}
where in the 3rd step we use the fact that $A$ is a comodule coalgebra and in the 4th step we use the fact that $A$ is a left $B$ comodule. So $(H, \blacktriangle, \epsilon_{H})$ is coassociative.
\end{proof}

From the proof above we can also see that even if $A$ is a Hopf coquasigroup, we can also get a Hopf coquasigroup $A\otimes B$, with the same coproduct, counit, and antipode as above.

\begin{lemma}\label{lemma. construct Hopf algebroid}
Let $(A, B, \phi, \delta)$ be a crossed comodule of Hopf coquasigroup. If $B$ is commutative and the image of $\phi$ belongs to the center of $A$, then $H=A\otimes B$ is a central Hopf algebroid over $B$ with the source, target and counit (of the bialgebroid structure) being bialgebra maps. Moreover, 
\begin{align*}
    (\Delta\otimes \Delta)\circ \blacktriangle=(\id_{H}\otimes \textup{flip}\otimes \id_{H})\circ (\blacktriangle\otimes_{B}\blacktriangle)\circ\Delta.
    \end{align*}
\end{lemma}
\begin{proof}
 The source and target maps $s, t: B\to H$ are given by $s(b):=\phi(\one{b})\ot \two{b}$, and $t(b):=1_{A}\ot b$, for any $b\in B$. The counit map $\epsilon: H\to B$ is defined to be $\epsilon(a\ot b):=\epsilon_{A}(a)b$, and the central Hopf algebroid coproduct is defined to be $\Delta(a\ot b):=(\one{a}\ot 1_{B})\ot_{B}(\two{a}\ot b)$. The antipode of the Hopf algebroid is given by $S(a\otimes b):=S_{A}(a)\phi(\one{b})\ot \two{b}$. Now we show all the structures above form a central Hopf algebroid structure on $H$.
First, we can see that $s, t$ are algebra maps. Now, we show $H$ is a $B$-coring, where the $B$-bimodule structure on $H$ is given by $b'\triangleright (a\ot b) \triangleleft b''=s(b')t(b'')(a\ot b)$ for $a\ot b\in H$, $b', b''\in B$.
So we have 
\begin{align*}
    \epsilon(b'\triangleright (a\ot b) \triangleleft b'')=&\epsilon(s(b')t(b'')(a\ot b))
    =\epsilon_{A}(\phi(\one{b'})a)\two{b'}b''b\\
    =&\epsilon_{B}(\one{b'})\epsilon_{A}(a)\two{b'}b''b
    =b'\epsilon(a\ot b)b'',
\end{align*}
where we use the fact that $\phi$ is a bialgebra map in the 3rd step. Clearly, $\epsilon$ is an algebra map from $A\ot B$ to $B$. We also have
\begin{align*}
(\epsilon\ot \epsilon)(\blacktriangle(a\ot b))=& (\epsilon\ot \epsilon)(\one{a}\ot\eins{\two{a}}\one{b}\ot\nul{\two{a}}\ot\two{b})\\
=&\epsilon_{A}(\one{a})\eins{\two{a}}\one{b}\ot \epsilon_{A}(\nul{\two{a}})\two{b}\\
    =&\eins{a}\one{b}\ot \epsilon_{A}(\nul{a})\two{b}\\
    =&\epsilon_{A}(a)\one{b}\ot \two{b}\\
    =&\Delta_{B}(\epsilon(a\ot b)),
\end{align*}
where for the 3rd step we use the fact that $A$ is a comodule algebra. So we can see that $\epsilon$ is a bialgebra map from $A\ot B$ to $B$.
We can also show $s$ and $t$ are also coalgebra maps
\begin{align*}
    \blacktriangle(s(b))=&\blacktriangle(\phi(\one{b})\ot \two{b})\\
    =&\one{\phi(\one{b})}\ot \eins{\two{\phi(\one{b})}}\one{\two{b}}\ot \nul{\two{\phi(\one{b})}}\ot\two{\two{b}}\\
    =&\phi(\one{\one{b}})\ot \eins{\phi(\two{\one{b}})}\one{\two{b}}\ot \nul{\phi(\two{\one{b}})}\ot\two{\two{b}}\\
    =&\phi(\one{\one{b}})\ot \eins{\phi(\one{\two{\one{b}}})}\two{\two{\one{b}}}\ot \nul{\phi(\one{\two{\one{b}}})}\ot\two{b}\\
    =&\phi(\one{\one{b}})\otimes \one{\one{\one{\two{\one{b}}}}}S_{B}(\two{\one{\two{\one{b}}}})\two{\two{\one{b}}}\otimes \phi(\two{\one{\one{\two{\one{b}}}}})\otimes \two{b}\\
    =&\phi(\one{\one{b}})\otimes \one{\two{\one{b}}}\otimes \phi(\two{\two{\one{b}}})\otimes \two{b}\\
    =&\phi(\one{\one{b}})\otimes \two{\one{b}}\otimes \phi(\one{\two{b}})\otimes \two{\two{b}}\\
    =&(s\otimes s)(\Delta_{B}(b)),
\end{align*}
where in the 4th and 7th steps we use Lemma \ref{lemma. coquasigroup}, and in the 5th step we use (\ref{cc3}). We also have
\begin{align*}
    \blacktriangle(t(b))=&\blacktriangle(1\ot b)
    =1\ot \one{b}\ot 1\ot \two{b}
    =(t\ot t)(\Delta_{B}(b)),
\end{align*}
for any $b\in B$. We can also show $\Delta$ is a $B$-bimodule map:
\begin{align*}
   \Delta(b'\triangleright(a\ot b))=&\Delta(\phi(\one{b'})a\ot \two{b'}b)\\
   =&(\one{\phi(\one{b'})}\one{a}\ot 1)\ot_{B}(\two{\phi(\one{b'})}\two{a}\ot \two{b'}b)\\
   =&(\phi(\one{\one{b'}})\one{a}\ot 1)\ot_{B}(\phi(\two{\one{b'}})\two{a}\ot \two{b'}b)\\
   =&(\phi(\one{b'})\one{a}\ot 1)\ot_{B}(\phi(\one{\two{b'}})\two{a}\ot \two{\two{b'}}b)\\
   =&(\phi(\one{b'})\one{a}\ot 1)\ot_{B} s(\two{b'})(\two{a}\ot b)\\
   =&t(\two{b'})(\phi(\one{b'})\one{a}\ot 1)\ot_{B}(\two{a}\ot b)\\
   =&(\phi(\one{b'})\one{a}\ot \two{b'})\ot_{B}(\two{a}\ot b)\\
   =&b' \triangleright \Delta(a\ot b),
\end{align*}
where in the fourth step we use Lemma \ref{lemma. coquasigroup}. We also have
\begin{align*}
    \Delta((a\ot b)\triangleleft b')=&\Delta(a\ot bb')
    =(\one{a}\ot 1)\ot_{B} (\two{a}\ot bb')
    =\Delta(a\ot b)\triangleleft b'
\end{align*}
for any $a\ot b\in A\ot B$ and $b'\in B$.
$\Delta$ is clearly coassociative and counital
by straightforward computation. Up to now we have already shown that $H$ is a $B$-coring. Clearly, $\Delta$ is also an algebra map from $H$ to $H\otimes_{B} H$. Since the image of $\phi$ belongs to the center of $A$, we can check that 
 $S(t(b')(a\ot b) s(b''))=t(b'')S(a\ot b)s(b')$ for any $a\ot b\in H$ and $b', b''\in B$. Indeed,
\begin{align*}
    S(t(b')(a\ot b) s(b''))=&S(a\phi(\one{b''})\ot b'b\two{b''})\\
    =&S_{A}(a\phi(\one{b''}))\phi(\one{b'}\one{b}\one{\two{b''}})\ot \two{b'}\two{b}\two{\two{b''}}\\
    =&S_{A}(a)\phi(\one{b})\phi(\one{b'})\ot b'' \two{b}\two{b'}\\
    =&t(b'')S(a\ot b)s(b'),
\end{align*}
where the 3rd step uses the fact that $B$ is commutative and its image of $\phi$ belongs to the center of $A$. Now, we can see that 
\begin{align*}
    S(\one{a}\ot 1)(\two{a}\ot b)=S_{A}(\one{a})\two{a}\ot b
    =(t\circ \epsilon)(a\ot b),
\end{align*}
and 
\begin{align*}
    (\one{a}\ot 1)S(\two{a}\ot b)=\one{a}S_{A}(\two{a})\phi(\one{b})\ot \two{b}
    =(s\circ \epsilon)(a\ot b)
\end{align*}
So $H$ is a central Hopf algebroid.

Let $h=a\otimes b\in H$, on the one hand
\begin{align*}
    &(\Delta\otimes \Delta)\circ \blacktriangle(h)
    =\one{a}\otimes 1\otimes_{B}\two{a}\otimes \eins{\three{a}}\one{b}\otimes \one{\nul{\three{a}}}\otimes 1\otimes_{B} \two{\nul{\three{a}}}\otimes \two{b},
\end{align*}
on the other hand
\begin{align*}
    (H\otimes &\textup{flip}\otimes H)\circ (\blacktriangle\otimes_{B}\blacktriangle)\circ(\Delta(h))\\
    =&\one{a}\otimes \eins{\two{a}}\otimes_{B}\three{a}\otimes \eins{\four{a}}\one{b}\otimes \nul{\two{a}}\otimes1\otimes_{B}\nul{\four{a}}\otimes \two{b}\\
    =&\one{a}\otimes 1\otimes_{B}\phi(\one{\eins{\two{a}}})\three{a}\otimes \two{\eins{\two{a}}}\eins{\four{a}}\one{b}\otimes \nul{\two{a}}\otimes 1\otimes_{B}\nul{\four{a}}\otimes \two{b}\\
    =&\one{a}\otimes 1\otimes_{B}\phi(\eins{\two{a}})\three{a}\otimes \eins{\nul{\two{a}}}\eins{\four{a}}\one{b}\otimes \nul{\nul{\two{a}}}\otimes 1\otimes_{B}\nul{\four{a}}\otimes \two{b}\\
    =&\one{a}\otimes 1\otimes_{B}\one{\two{a}}S_{A}(\three{\two{a}})\three{a}\otimes \eins{\two{\two{a}}}\eins{\four{a}}\one{b}\otimes \nul{\two{\two{a}}}\otimes 1\otimes_{B}\nul{\four{a}}\otimes \two{b}\\
    =&\one{a}\otimes 1\otimes_{B}\two{a}\otimes \eins{\three{a}}\eins{\four{a}}\one{b}\otimes \nul{\three{a}}\otimes 1\otimes_{B}\nul{\four{a}}\otimes \two{b}\\
    =&\one{a}\otimes 1\otimes_{B}\two{a}\otimes \eins{\three{a}}\one{b}\otimes \one{\nul{\three{a}}}\otimes 1\otimes_{B} \two{\nul{\three{a}}}\otimes \two{b},
\end{align*}
where  the second step  uses the balanced tensor product over $B$, the fourth step uses (\ref{cc4}) and  the last step uses the fact that $A$ is a comodule coalgebra of $B$.
\end{proof}

\subsection{Quasi coassociative Hopf coquasigroup}\label{sec. Quasi coassociative Hopf coquasigroup}

 In this section we will construct a coherent Hopf 2-algebra in terms of a crossed comodule of Hopf coquasigroup. First we define quasi coassociative Hopf coquasigroups, which can be viewed as a quantization of quasiassociative quasigroups.

\begin{definition}\label{def. quasi coassociative}

A coassociative pair $(C, B, \phi)$ is called \textit{quasi coassociative} if:
\begin{itemize}
   \item [(1)] $\phi: B\to C$ is a surjective morphism of Hopf coquasigroups.
    \item [(2)] For any $i\in I_{B}:=\ker(\phi)$,
    \begin{align}\label{equ. quasi associative}
        \begin{cases}
          \one{\one{i}}S_{B}(\two{i})\otimes \two{\one{i}}&\in B\otimes I_{B},\\
          \one{i}S_{B}(\two{\two{i}})\otimes \one{\two{i}}&\in  B\otimes I_{B}.
        \end{cases}
    \end{align}
    \item [(3)] $I_{B}\subseteq \ker(\beta)$, where $\beta: B\to B\otimes B\otimes B$ is the comultiplicative  coassociator (\ref{coassociator}).
\end{itemize}
\end{definition}
 Since $\phi$ is surjective, any element in $C$ can be given by $[c]:=\phi(c)$ for some element $c\in B$. If $B$ is a quasi coassociative, by (\ref{equ. quasi associative}) there is a linear map $\Ad: C\to B\otimes C$,
$\Ad([c]):=\one{c}S_{B}(\three{c})\otimes [\two{c}] :=\one{\one{c}}S_{B}(\two{c})\otimes[\one{\two{c}}] = \one{c}S_{B}(\two{\two{c}})\otimes[\two{\one{c}}]$ for any $[c]\in C$ (the last equality hold because of Lemma \ref{lemma. coquasigroup}). For any $b\in B$, there is an important result of Lemma  \ref{lemma. coquasigroup}:
\begin{align}\label{equ. quasigroup relation 1}
    \one{\one{b}}S_{B}(\three{\one{b}})\two{b}\otimes [\two{\one{b}}]=\one{\one{\one{b}}}S_{B}(\two{\one{b}})\two{b}\otimes [\two{\one{\one{b}}}]=\one{b}\otimes [\two{b}].
\end{align}
Similarly,
\begin{align}\label{equ. quasigroup relation 2}
    S_{B}(\one{b})\one{\two{b}}S_{B}(\three{\two{b}})\otimes [\two{\two{b}}]=S_{B}(\two{b})\otimes [\one{b}].
\end{align}
Since $I_{B}\subseteq \ker(\beta)$, the comultiplicative coassociator $\beta$ factors through $C$, namely, there is a linear map $\tilde{\beta}: C\to B\otimes B\otimes B$ given by $\tilde{\beta}([c]):=\beta(c)=c^{\hat{1}}\otimes c^{\hat{2}}\otimes c^{\hat{3}}$.

\begin{lemma}\label{lemma. construct crossed module}
Let  $(C, B, \phi)$ be  quasi coassociative. If $B$ is commutative, then $(C, B, \phi,\Ad)$ is a crossed comodule of Hopf coquasigroup. 

\end{lemma}
\begin{proof}

Since 
\[ \one{\one{c}}S_{B}(\two{\two{c}})\otimes \two{\one{c}}S_{B}(\one{\two{c}})=c\o\o  S_{B}(c\t)\ot c\o\t\o S_{B}(c\o\t\t),\]
for any $c\in B$, by Lemma \ref{lemma. coquasigroup}, we can prove $\Ad$ is a coaction. Indeed,
\begin{align*}\label{equ. Ad as comodule map}
   \one{\one{c}}S_{B}(\two{\three{c}})\otimes \two{\one{c}}S_{B}(\one{\three{c}})\otimes [\two{c}]=c\o\o S_{B}(c\t)\ot c\o\t\o S_{B}(c\o\t\th)\ot [c\o\t\t].
\end{align*}
We can see that $\Ad$ is an algebra map, since $B$ is commutative. Now let's show that $\Ad$ is a comodule coalgebra map.
On the one hand 
\begin{align*}
    \eins{[c]}\ot \one{\nul{[c]}}\ot \two{\nul{[c]}}
    =\one{c}S_{B}(\three{c})\otimes \one{[\two{c}]}\otimes \two{[\two{c}]}.
\end{align*}
On the other hand
\begin{align*}
    \eins{\one{[c]}}&\eins{\two{[c]}}\ot \nul{\one{[c]}}\ot \nul{\two{[c]}}\\
    =&\one{\one{c}}S_{B}(\three{\one{c}})\one{\two{c}}S_{B}(\three{\two{c}})\otimes [\two{\one{c}}]\otimes [\two{\two{c}}]\\
    =&\one{c}S_{B}(\three{c})\otimes \one{[\two{c}]}\otimes \two{[\two{c}]},
\end{align*}
where the last step uses Lemma \ref{lemma. coquasigroup}. And
\begin{align*}
    \epsilon_{C}([c])=\epsilon_{B}(c)=\one{c}S_{B}(\three{c})\epsilon_{C}[\two{c}]=\eins{[c]}\epsilon_{C}(\nul{[c]}).
\end{align*}
We can see (\ref{cc3}) and (\ref{cc4}) are given by the definition of $\Ad$.

\end{proof}

Now we want to construct a coherent 2-group in terms of the crossed comodule $(C, B, \phi, \Ad)$ we just considered above. In the following, we always assume $B$ is commutative. Compare to Definition \ref{def. coherent Hopf-2-algebra}, the first Hopf coquasigroup is $B$. The second Hopf coquasigroup is $H:=C\otimes B$. By Lemma \ref{lemma. construct Hopf coquasigroup}, as a Hopf coquasigroup the coproduct, counit and antipode are given by the following:
\begin{align}
   \blacktriangle([c]\otimes b)&:=\one{[c]}\ot \eins{\two{[c]}}\one{b}\ot \nul{\two{[c]}}\ot \two{b}, \\
    \epsilon_{H}([c]\otimes b)&:=\epsilon_{B}(c)\epsilon_{B}(b),\\
    S_{H}([c]\otimes b)&:=[S_{B}(\two{\one{c}})]\otimes S_{B}(\one{c})\two{\two{c}}S_{B}(b)=S_{C}(\nul{[c]})\otimes S_{B}(\eins{[c]}b).
\end{align}
By Lemma \ref{lemma. construct Hopf algebroid} with the source and target maps $s, t: B\to H$ are given by
\begin{align}\label{equ. source and target map}
    s(b):=[\one{b}]\ot\two{b},\quad t(b):=1_{C}\otimes b,
\end{align}
for any $b\in B$. The Hopf algebroid coproduct is given by
\begin{align}\label{equ. coproduct of bialgebroid}
  \Delta([c]\otimes b):=([\one{c}]\otimes 1_{B})\otimes_{B}([\two{c}]\otimes b),  
\end{align} and the counit is given by 
\begin{align}\label{equ. counit of bialgebroid}
    \epsilon([c]\otimes b):=\epsilon_{B}(c)b.
\end{align}
The Hopf algebroid antipode is
\begin{align}\label{equ. antipode of bialgebroid}
    S([c]\otimes b):=[S_{B}(c)\one{b}]\otimes \two{b}.
\end{align}
We define the coassociator $\alpha: H\to B\otimes B\otimes B$ by
\begin{align}\label{def. coassociator}
    \alpha([c]\otimes b):=\beta(c)(\one{\one{b}}\otimes\two{\one{b}}\otimes \two{b})=c^{\hat{1}}\one{\one{b}}\otimes c^{\hat{2}}\two{\one{b}}\otimes c^{\hat{3}}\two{b}.
\end{align}
 By using (\ref{equ. relation of coassociator}), we can check condition (iv) of Definition \ref{def. coherent Hopf-2-algebra}:
\begin{align*}
    \alpha(t(b))=\one{\one{b}}\otimes\two{\one{b}}\otimes \two{b}
\end{align*}
and
\begin{align*}
    \alpha(s(b))=&(\one{b})^{\hat{1}}\one{\one{\two{b}}}\otimes (\one{b})^{\hat{2}}\two{\one{\two{b}}}\otimes (\one{b})^{\hat{3}}\two{\two{b}}\\
    =&\one{b}\otimes \one{\two{b}}\otimes \two{\two{b}},
\end{align*}
where the second step uses (\ref{equ. relation of coassociator}). Now let's check (v) and (vi) of Definition \ref{def. coherent Hopf-2-algebra}.

\begin{lemma}\label{lemma. natural of coassociator}
The coassociator $\alpha:H\to B\ot B\ot B$ satisfies:
\[ ((s\otimes s\otimes s)\circ \alpha)\star ((\blacktriangle\otimes \id_{H})\circ \blacktriangle)=((\id_{H}\otimes \blacktriangle)\circ \blacktriangle)\star ((t\otimes t\otimes t)\circ \alpha).\]
 More precisely,
for any $h\in H$, we have
\begin{align*}
        &s(\uone{h}\yi{})\one{\one{\utwo{h}}}\ot s(\uone{h}\er{})\two{\one{\utwo{h}}}\ot s(\uone{h}\san{})\two{\utwo{h}}\\
        =&\one{\uone{h}}t(\yi{\utwo{h}})\ot \one{\two{\uone{h}}}t(\er{\utwo{h}})\ot \two{\two{\uone{h}}}t(\san{\utwo{h}}).
    \end{align*}
\end{lemma}
\begin{proof}
Let $h=[c]\otimes b$. The left hand side of the equation above is:
\begin{align*}
    &s((\one{c})^{\hat{1}})([\one{\one{\two{c}}}]\otimes \eins{[\two{\one{\two{c}}}]}\one{\eins{[\two{\two{c}}]}}\one{\one{b}})
    \otimes s((\one{c})^{\hat{2}})(\nul{[\two{\one{\two{c}}}]}\otimes \two{\eins{[\two{\two{c}}]}}\two{\one{b}})\\
    &\otimes s((\one{c})^{\hat{3}}) (\nul{[\two{\two{c}}]}\otimes \two{b}),
\end{align*}
while the right hand side of the equation is
\begin{align*}
    &[\one{\one{c}}]\otimes \eins{[\two{\one{c}}]}(\two{c})^{\hat{1}}\one{\one{b}}
    \otimes\one{\nul{[\two{\one{c}}]}}\otimes \eins{\two{\nul{[\two{\one{c}}]}}}(\two{c})^{\hat{2}}\two{\one{b}}\\
    &\otimes\nul{\two{\nul{[\two{\one{c}}]}}}\otimes (\two{c})^{\hat{3}}\two{b}.
\end{align*}
So it is sufficient to show
\begin{align*}
    s((\one{c})^{\hat{1}})&([\one{\one{\two{c}}}]\otimes \eins{[\two{\one{\two{c}}}]}\one{\eins{[\two{\two{c}}]}})
    \otimes s((\one{c})^{\hat{2}})(\nul{[\two{\one{\two{c}}}]}\otimes \two{\eins{[\two{\two{c}}]}})\\
    &\otimes s((\one{c})^{\hat{3}}) (\nul{[\two{\two{c}}]}\otimes 1)\\
    =    &[\one{\one{c}}]\otimes \eins{[\two{\one{c}}]}(\two{c})^{\hat{1}}
    \otimes\one{\nul{[\two{\one{c}}]}}\otimes \eins{\two{\nul{[\two{\one{c}}]}}}(\two{c})^{\hat{2}}
    \otimes\nul{\two{\nul{[\two{\one{c}}]}}}\otimes (\two{c})^{\hat{3}}.
\end{align*}
By the definition of Hopf coquasigroup, this is equivalent to
\begin{align*}
    s((\one{\one{\one{c}}})^{\hat{1}})&([\one{\one{\two{\one{\one{c}}}}}]\otimes \eins{[\two{\one{\two{\one{\one{c}}}}}]}\one{\eins{[\two{\two{\one{\one{c}}}}]}}
    \one{\one{\two{\one{c}}}}\one{\one{S_{B}(\two{c})}})\\
    &\otimes (s((\one{\one{\one{c}}})^{\hat{2}})\nul{[\two{\one{\two{\one{\one{c}}}}}]}\otimes \two{\eins{[\two{\two{\one{\one{c}}}}]}}\two{\one{\two{\one{c}}}}\two{\one{S_{B}(\two{c})}})\\
    &\otimes (s((\one{\one{\one{c}}})^{\hat{3}}) \nul{[\two{\two{\one{\one{c}}}}]}\otimes\two{\two{\one{c}}}\two{S_{B}(\two{c})}) \\
    =&[\one{\one{\one{\one{c}}}}]\otimes \eins{[\two{\one{\one{\one{c}}}}]}(\two{\one{\one{c}}})^{\hat{1}}\one{\one{\two{\one{c}}}}\one{\one{S_{B}(\two{c})}}\\
    &\otimes\one{\nul{[\two{\one{\one{\one{c}}}}]}}\otimes \eins{\two{\nul{[\two{\one{\one{\one{c}}}}]}}}(\two{\one{\one{c}}})^{\hat{2}}\two{\one{\two{\one{c}}}}\two{\one{S_{B}(\two{c})}}\\
    &\otimes\nul{\two{\nul{[\two{\one{\one{\one{c}}}}]}}}\otimes (\two{\one{\one{c}}})^{\hat{3}}\two{\two{\one{c}}}\two{S_{B}(\two{c})}.
\end{align*}

Thus it is sufficient to show
\begin{align*}
        s((\one{\one{c}})^{\hat{1}})&([\one{\one{\two{\one{c}}}}]\otimes \eins{[\two{\one{\two{\one{c}}}}]}\one{\eins{[\two{\two{\one{c}}}]}}\one{\one{\two{c}}})\\
    &\otimes s((\one{\one{c}})^{\hat{2}})(\nul{[\two{\one{\two{\one{c}}}}]}\otimes \two{\eins{[\two{\two{\one{c}}}]}}\two{\one{\two{c}}})
    \otimes s((\one{\one{c}})^{\hat{3}}) (\nul{[\two{\two{\one{c}}}]}\otimes\two{\two{c}}) \\
    =&[\one{\one{\one{c}}}]\otimes \eins{[\two{\one{\one{c}}}]}(\two{\one{c}})^{\hat{1}}\one{\one{\two{c}}}\\
    &\otimes\one{\nul{[\two{\one{\one{c}}}]}}\otimes \eins{\two{\nul{[\two{\one{\one{c}}}]}}}(\two{\one{c}})^{\hat{2}}\two{\one{\two{c}}}
    \otimes\nul{\two{\nul{[\two{\one{\one{c}}}]}}}\otimes (\two{\one{c}})^{\hat{3}}\two{\two{c}}.
\end{align*}

The left hand side is
\begin{align*}
    s((\one{\one{c}})^{\hat{1}})&([\one{\one{\two{\one{c}}}}]\otimes \eins{[\two{\one{\two{\one{c}}}}]}\one{\eins{[\two{\two{\one{c}}}]}}\one{\one{\two{c}}})\\
    &\otimes s((\one{\one{c}})^{\hat{2}})(\nul{[\two{\one{\two{\one{c}}}}]}\otimes \two{\eins{[\two{\two{\one{c}}}]}}\two{\one{\two{c}}})
    \otimes s((\one{\one{c}})^{\hat{3}})( \nul{[\two{\two{\one{c}}}]}\otimes\two{\two{c}}) \\
    =&s((\one{c})^{\hat{1}})([\one{\one{\two{c}}}]\otimes \eins{[\two{\one{\two{c}}}]}\one{\eins{[\one{\one{\two{\two{c}}}}]}}\one{\two{\one{\two{\two{c}}}}})\\
    &\otimes s((\one{c})^{\hat{2}})(\nul{[\two{\one{\two{c}}}]}\otimes \two{\eins{[\one{\one{\two{\two{c}}}}]}}\two{\two{\one{\two{\two{c}}}}})
    \otimes s((\one{c})^{\hat{3}}) (\nul{[\one{\one{\two{\two{c}}}}]}\otimes\two{\two{\two{c}}})\\
    =&s((\one{c})^{\hat{1}})([\one{\one{\one{\two{c}}}}]\otimes \eins{[\two{\one{\one{\two{c}}}}]}\one{\two{\one{\two{c}}}})\\
    &\otimes s((\one{c})^{\hat{2}})(\nul{[\two{\one{\one{\two{c}}}}]}\otimes \two{\two{\one{\two{c}}}})
    \otimes s((\one{c})^{\hat{3}}) ([\one{\two{\two{c}}}]\otimes\two{\two{\two{c}}})\\
    =&s((\one{c})^{\hat{1}})([\one{\one{\two{c}}}]\otimes \eins{[\one{\one{\two{\one{\two{c}}}}}]}\two{\one{\two{\one{\two{c}}}}})\\
    &\otimes s((\one{c})^{\hat{2}})(\nul{[\one{\one{\two{\one{\two{c}}}}}]}\otimes \two{\two{\one{\two{c}}}})
    \otimes s((\one{c})^{\hat{3}}) ([\one{\two{\two{c}}}]\otimes\two{\two{\two{c}}})\\
    =&s((\one{c})^{\hat{1}})([\one{\one{\two{c}}}]\otimes \one{\one{\two{\one{\two{c}}}}})
    \otimes s((\one{c})^{\hat{2}})([\two{\one{\two{\one{\two{c}}}}}]\otimes \two{\two{\one{\two{c}}}})\\
    &\otimes s((\one{c})^{\hat{3}}) ([\one{\two{\two{c}}}]\otimes\two{\two{\two{c}}})\\
    =&[\one{((\one{c})^{\hat{1}})}\one{\one{\one{\two{c}}}}]\otimes \two{((\one{c})^{\hat{1}})}\two{\one{\one{\two{c}}}}
    \otimes[\one{((\one{c})^{\hat{2}})}\one{\two{\one{\two{c}}}}]\otimes \two{((\one{c})^{\hat{2}})}\two{\two{\one{\two{c}}}}\\
    &\otimes[\one{((\one{c})^{\hat{3}})}\one{\two{\two{c}}}]\otimes \two{((\one{c})^{\hat{3}})}\two{\two{\two{c}}}\\
    =&[\one{\one{c}}]\otimes \two{\one{c}}\otimes [\one{\one{\two{c}}}]\otimes \two{\one{\two{c}}}\otimes [\one{\two{\two{c}}}]\otimes \two{\two{\two{c}}}
\end{align*}
where for the 1st, 2nd, 3rd, 5th steps we use Lemma \ref{lemma. coquasigroup} and the fact that $\beta$ factors through $C$, in the 2nd, 4th steps we use (\ref{equ. quasigroup relation 1}), and the last step uses (\ref{equ. relation of coassociator}). The right hand side is:

\begin{align*}
[\one{\one{\one{c}}}]&\otimes \eins{[\two{\one{\one{c}}}]}(\two{\one{c}})^{\hat{1}}\one{\one{\two{c}}}
\otimes\one{\nul{[\two{\one{\one{c}}}]}}\otimes \eins{\two{\nul{[\two{\one{\one{c}}}]}}}(\two{\one{c}})^{\hat{2}}\two{\one{\two{c}}}\\
&\otimes\nul{\two{\nul{[\two{\one{\one{c}}}]}}}\otimes (\two{\one{c}})^{\hat{3}}\two{\two{c}}\\
=&[\one{\one{c}}]\otimes \eins{[\two{\one{c}}]}(\one{\two{c}})^{\hat{1}}\one{\one{\two{\two{c}}}}
\otimes\one{\nul{[\two{\one{c}}]}}\otimes \eins{\two{\nul{[\two{\one{c}}]}}}(\one{\two{c}})^{\hat{2}}\two{\one{\two{\two{c}}}}\\
&\otimes\nul{\two{\nul{[\two{\one{c}}]}}}\otimes (\one{\two{c}})^{\hat{3}}\two{\two{\two{c}}}\\
=&[\one{\one{c}}]\otimes \eins{[\two{\one{c}}]}\one{\two{c}}
\otimes\one{\nul{[\two{\one{c}}]}}\otimes \eins{\two{\nul{[\two{\one{c}}]}}}\one{\two{\two{c}}}\\
&\otimes\nul{\two{\nul{[\two{\one{c}}]}}}\otimes \two{\two{\two{c}}}\\
=&[\one{c}]\otimes \eins{[\one{\one{\two{c}}}]}\two{\one{\two{c}}}
\otimes\one{\nul{[\one{\one{\two{c}}}]}}\otimes \eins{\two{\nul{[\one{\one{\two{c}}}]}}}\one{\two{\two{c}}}\\
&\otimes\nul{\two{\nul{[\one{\one{\two{c}}}]}}}\otimes \two{\two{\two{c}}}\\
=&[\one{c}]\otimes \one{\one{\two{c}}}
\otimes[\one{\two{\one{\two{c}}}}]\otimes \eins{[\two{\two{\one{\two{c}}}}]}\one{\two{\two{c}}}
\otimes\nul{[\two{\two{\one{\two{c}}}}]}\otimes \two{\two{\two{c}}}\\
=&[\one{c}]\otimes \one{\two{c}}
\otimes[\one{\two{\two{c}}}]\otimes \eins{[\one{\one{\two{\two{\two{c}}}}}]}\two{\one{\two{\two{\two{c}}}}}
\otimes\nul{[\one{\one{\two{\two{\two{c}}}}}]}\otimes\two{\two{\two{\two{c}}}}\\
=& [\one{c}]\otimes \one{\two{c}}\otimes [\one{\two{\two{c}}}]\otimes \one{\one{\two{\two{\two{c}}}}}\otimes[\two{\one{\two{\two{\two{c}}}}}]\otimes \two{\two{\two{\two{c}}}}\\
=&[\one{\one{c}}]\otimes \two{\one{c}}\otimes [\one{\one{\two{c}}}]\otimes \two{\one{\two{c}}}\otimes [\one{\two{\two{c}}}]\otimes \two{\two{\two{c}}},
\end{align*}
where in the 1st, 3rd, 5th and last steps we use Lemma \ref{lemma. coquasigroup} and the fact that $\beta$ factors through $C$, the 2nd step uses (\ref{equ. relation of coassociator}), and the 4th and 6th steps use (\ref{equ. quasigroup relation 1}).

\end{proof}

\begin{lemma}\label{lemma. 3-cocycle relation} The coassociator  $\alpha: H\to B\otimes B\otimes B$ satisfies the 3-cocycle condition:
\[  ((\epsilon\otimes \alpha)\circ \blacktriangle)\star ((\id_{B}\otimes \Delta_{B}\otimes \id_{B})\circ \alpha)\star ((\alpha\otimes \epsilon)\circ \blacktriangle)
                =((\id_{B}\otimes \id_{B}\otimes \Delta_{B})\circ \alpha)\star ((\Delta_{B}\otimes \id_{B}\otimes \id_{B})\circ \alpha). \]
                More precisely, for any $h\in H$.
\begin{align*}
        &\yi{\uone{h}}\one{\yi{\utwo{h}}}\ot\er{\uone{h}}\two{\yi{\utwo{h}}}\ot \one{\san{\uone{h}}}\er{\utwo{h}}\ot\two{\san{\uone{h}}}\san{\utwo{h}}\\
        =&\epsilon(\one{\uone{h}})\yi{\utwo{h}}\yi{\one{\uthree{h}}}\ot \yi{\two{\uone{h}}}\one{\er{\utwo{h}}}\er{\one{\uthree{h}}}\ot \er{\two{\uone{h}}}\two{\er{\utwo{h}}}\san{\one{\uthree{h}}}\ot \san{\two{\uone{h}}}\san{\utwo{h}}\epsilon(\two{\uthree{h}}).
    \end{align*}
   
\end{lemma}
\begin{proof}
Let $h=[c]\otimes b$, the left hand side is
\begin{align*}
    (\one{c})^{\hat{1}}\one{(\two{c})^{\hat{1}}}\one{\one{\one{b}}}\otimes(\one{c})^{\hat{2}}\two{ (\two{c})^{\hat{1}}}\two{\one{\one{b}}}\otimes \one{ (\one{c})^{\hat{3}}} (\two{c})^{\hat{2}}\two{\one{b}}\otimes\two{ (\one{c})^{\hat{3}}} (\two{c})^{\hat{3}}\two{b},
\end{align*}
the right hand side is
\begin{align*}
    &\one{\one{\one{c}}}S_{B}(\two{\one{c}}) (c\t\o)^{\hat{1}} (c\t\t)^{\hat{1}}\one{\one{\one{b}}}\otimes  (\two{\one{\one{c}}})^{\hat{1}}\one{ (c\t\o)^{\hat{2}}} (c\t\t)^{\hat{2}}\two{\one{\one{b}}}\\
    \otimes& (\two{\one{\one{c}}})^{\hat{2}}\two{ (c\t\o)^{\hat{2}}} (c\t\t)^{\hat{3}}\two{\one{b}}\otimes  (\two{\one{\one{c}}})^{\hat{3}} (c\t\o)^{\hat{3}}\two{b}.
\end{align*}

We first observe that 
\begin{align*}
        (\one{c})^{\hat{1}}&\one{(\two{c})^{\hat{1}}}\otimes(\one{c})^{\hat{2}}\two{ (\two{c})^{\hat{1}}}\otimes \one{ (\one{c})^{\hat{3}}} (\two{c})^{\hat{2}}\otimes\two{ (\one{c})^{\hat{3}}} (\two{c})^{\hat{3}}\\
        =&(\one{\one{\one{c}}})^{\hat{1}}\one{(\two{\one{\one{c}}})^{\hat{1}}}\one{\one{\one{\two{\one{c}}}}}\one{\one{\one{S_{B}(\two{c})}}}\\
        &\otimes(\one{\one{\one{c}}})^{\hat{2}}\two{(\two{\one{\one{c}}})^{\hat{1}}}\two{\one{\one{\two{\one{c}}}}}\two{\one{\one{S_{B}(\two{c})}}}\\
        &\otimes \one{(\one{\one{\one{c}}})^{\hat{3}}} (\two{\one{\one{c}}})^{\hat{2}}\two{\one{\two{\one{c}}}}\two{\one{S_{B}(\two{c})}}\\
        &\otimes\two{ (\one{\one{\one{c}}})^{\hat{3}}} (\two{\one{\one{c}}})^{\hat{3}}\two{\two{\one{c}}}\two{S_{B}(\two{c})}
\end{align*}
and
\begin{align*}
      \one{\one{\one{c}}}&S_{B}(\two{\one{c}}) (c\t\o)^{\hat{1}} (c\t\t)^{\hat{1}}\otimes  (\two{\one{\one{c}}})^{\hat{1}}\one{ (c\t\o)^{\hat{2}}} (c\t\t)^{\hat{2}}\\
    &\otimes (\two{\one{\one{c}}})^{\hat{2}}\two{ (c\t\o)^{\hat{2}}} (c\t\t)^{\hat{3}}\otimes  (\two{\one{\one{c}}})^{\hat{3}} (c\t\o)^{\hat{3}}\\
    =&\one{\one{\one{\one{\one{c}}}}}S_{B}(\two{\one{\one{\one{c}}}}) (c\o\o\t\o)^{\hat{1}} (c\o\o\t\t)^{\hat{1}}\one{\one{\one{\two{\one{c}}}}}\one{\one{\one{S_{B}(\two{c})}}}\\
    &\otimes  (\two{\one{\one{\one{\one{c}}}}})^{\hat{1}}\one{ (c\o\o\t\o)^{\hat{2}}} (c\o\o\t\t)^{\hat{2}}\two{\one{\one{\two{\one{c}}}}}\two{\one{\one{S_{B}(\two{c})}}}\\
    &\otimes (\two{\one{\one{\one{\one{c}}}}})^{\hat{2}}\two{ (c\o\o\t\o)^{\hat{2}}} (c\o\o\t\t)^{\hat{3}}\two{\one{\two{\one{c}}}}\two{\one{S_{B}(\two{c})}}\\
    &\otimes  (\two{\one{\one{\one{\one{c}}}}})^{\hat{3}} (c\o\o\t\o)^{\hat{3}}\two{\two{\one{c}}}\two{S_{B}(\two{c})}.
\end{align*}
Thus to show this lemma it is sufficient to show: 
\begin{align*}
    (\one{\one{c}})^{\hat{1}}&\one{(\two{\one{c}})^{\hat{1}}}\one{\one{\one{\two{c}}}}\otimes(\one{\one{c}})^{\hat{2}}\two{(\two{\one{c}})^{\hat{1}}}\two{\one{\one{\two{c}}}}\\
    &\otimes \one{(\one{\one{c}})^{\hat{3}}} (\two{\one{c}})^{\hat{2}}\two{\one{\two{c}}}\otimes\two{ (\one{\one{c}})^{\hat{3}}} (\two{\one{c}})^{\hat{3}}\two{\two{c}}\\
    =&\one{\one{\one{\one{c}}}}S_{B}(\two{\one{\one{c}}}) (c\o\t\o)^{\hat{1}} (c\o\t\t)^{\hat{1}}\one{\one{\one{\two{c}}}}\\
    &\otimes  (\two{\one{\one{\one{c}}}})^{\hat{1}}\one{ (c\o\t\o)^{\hat{2}}} (c\o\t\t)^{\hat{2}}\two{\one{\one{\two{c}}}}\\
    &\otimes (\two{\one{\one{\one{c}}}})^{\hat{2}}\two{ (c\o\t\o)^{\hat{2}}} (c\o\t\t)^{\hat{3}}\two{\one{\two{c}}}\\
    &\otimes  (\two{\one{\one{\one{c}}}})^{\hat{3}} (c\o\t\o)^{\hat{3}}\two{\two{c}}.    
\end{align*}
Since $\ker(\phi)\subseteq \ker(\beta)$, the left hand side of the above equation becomes
\begin{align*}
    (\one{c})^{\hat{1}}&\one{(\one{\two{c}})^{\hat{1}}}\one{\one{\one{\two{\two{c}}}}}\otimes(\one{c})^{\hat{2}}\two{(\one{\two{c}})^{\hat{1}}}\two{\one{\one{\two{\two{c}}}}}\\
    &\otimes \one{(\one{c})^{\hat{3}}} (\one{\two{c}})^{\hat{2}}\two{\one{\two{\two{c}}}}\otimes\two{ (\one{c})^{\hat{3}}} (\one{\two{c}})^{\hat{3}}\two{\two{\two{c}}}\\
    =&(\one{c})^{\hat{1}}\one{\one{\two{c}}}\otimes(\one{c})^{\hat{2}}\two{\one{\two{c}}}
    \otimes \one{(\one{c})^{\hat{3}}} \one{\two{\two{c}}}\otimes\two{ (\one{c})^{\hat{3}}} \two{\two{\two{c}}}\\
    =&\one{c}\otimes\one{\two{c}}\otimes \one{\two{\two{c}}}\otimes\two{\two{\two{c}}},
\end{align*}
where the first step uses (\ref{equ. relation of coassociator}).
By  Lemma \ref{lemma. coquasigroup}  and the fact that $\beta$ factors through $C$, the right hand side becomes
\begin{align*}
    &\one{\one{\one{c}}}S_{B}(\two{\one{c}}) (\one{\two{c}})^{\hat{1}} (\one{\two{\two{c}}})^{\hat{1}}\one{\one{\one{\two{\two{\two{c}}}}}}
    \otimes  (\two{\one{\one{c}}})^{\hat{1}}\one{ (\one{\two{c}})^{\hat{2}}} (\one{\two{\two{c}}})^{\hat{2}}\two{\one{\one{\two{\two{\two{c}}}}}}\\
    &\otimes (\two{\one{\one{c}}})^{\hat{2}}\two{ (\one{\two{c}})^{\hat{2}}} (\one{\two{\two{c}}})^{\hat{3}}\two{\one{\two{\two{\two{c}}}}}
    \otimes  (\two{\one{\one{c}}})^{\hat{3}} (\one{\two{c}})^{\hat{3}}\two{\two{\two{\two{c}}}} \\
    =&\one{\one{\one{c}}}S_{B}(\two{\one{c}}) (\one{\two{c}})^{\hat{1}} (\one{\one{\two{\two{c}}}})^{\hat{1}}\one{\one{\two{\one{\two{\two{c}}}}}}
    \otimes  (\two{\one{\one{c}}})^{\hat{1}}\one{ (\one{\two{c}})^{\hat{2}}} (\one{\one{\two{\two{c}}}})^{\hat{2}}\two{\one{\two{\one{\two{\two{c}}}}}}\\
    &\otimes (\two{\one{\one{c}}})^{\hat{2}}\two{ (\one{\two{c}})^{\hat{2}}} (\one{\one{\two{\two{c}}}})^{\hat{3}}\two{\two{\one{\two{\two{c}}}}}
    \otimes  (\two{\one{\one{c}}})^{\hat{3}} (\one{\two{c}})^{\hat{3}}\two{\two{\two{c}}} \\
    =&\one{\one{\one{c}}}S_{B}(\two{\one{c}}) (\one{\two{c}})^{\hat{1}} \one{\one{\two{\two{c}}}}
    \otimes  (\two{\one{\one{c}}})^{\hat{1}}\one{ (\one{\two{c}})^{\hat{2}}} \one{\two{\one{\two{\two{c}}}}}\\
    &\otimes (\two{\one{\one{c}}})^{\hat{2}}\two{ (\one{\two{c}})^{\hat{2}}} \two{\two{\one{\two{\two{c}}}}}
    \otimes  (\two{\one{\one{c}}})^{\hat{3}} (\one{\two{c}})^{\hat{3}}\two{\two{\two{c}}} \\
    =&\one{\one{\one{c}}}S_{B}(\two{\one{c}})  \one{\two{c}}
    \otimes  (\two{\one{\one{c}}})^{\hat{1}} \one{\one{\two{\two{c}}}}
    \otimes (\two{\one{\one{c}}})^{\hat{2}}\two{\one{\two{\two{c}}}}
    \otimes  (\two{\one{\one{c}}})^{\hat{3}} \two{\two{\two{c}}} \\
    =&\one{\one{\one{\one{c}}}}S_{B}(\two{\one{\one{c}}})  \two{\one{c}}
    \otimes  (\two{\one{\one{\one{c}}}})^{\hat{1}} \one{\one{\two{c}}}
    \otimes (\two{\one{\one{\one{c}}}})^{\hat{2}}\two{\one{\two{c}}}
    \otimes  (\two{\one{\one{\one{c}}}})^{\hat{3}} \two{\two{c}} \\
    =&\one{\one{c}}
    \otimes  (\two{\one{c}})^{\hat{1}} \one{\one{\two{c}}}
    \otimes (\two{\one{c}})^{\hat{2}}\two{\one{\two{c}}}
    \otimes  (\two{\one{c}})^{\hat{3}} \two{\two{c}} \\
    =&\one{c}
    \otimes  (\one{\two{c}})^{\hat{1}} \one{\one{\two{\two{c}}}}
    \otimes (\one{\two{c}})^{\hat{2}}\two{\one{\two{\two{c}}}}
    \otimes  (\one{\two{c}})^{\hat{3}} \two{\two{\two{c}}} \\
    =&\one{c}\otimes\one{\two{c}}\otimes \one{\two{\two{c}}}\otimes\two{\two{\two{c}}},
\end{align*}
where in the 1st, 4th and 6th steps we use Lemma \ref{lemma. coquasigroup}  and the fact that $\beta$ factors through $C$, for
the 2nd, 3rd and last steps we use (\ref{equ. relation of coassociator}), and the 5th step uses (\ref{equ. quasigroup relation 1}) (since $\ker(\phi)\subseteq \ker(\beta)$).


\end{proof}
As a result of Lemmas  \ref{lemma. construct Hopf coquasigroup}, \ref{lemma. construct Hopf algebroid}, \ref{lemma. construct crossed module}, \ref{lemma. natural of coassociator} and \ref{lemma. 3-cocycle relation}
we have
\begin{theorem}\label{thm. main theorem}
If $(C, B, \phi)$ is quasi coassociative and $B$ is commutative, then $H=C\otimes B$ is a coherent Hopf 2-algebra with the structure maps given by:
\begin{align*}
    ([c]\ot b)([c']\ot b')=&[cc']\ot bb',\\
    \blacktriangle([c]\ot b)=&\one{[c]}\ot \eins{\two{[c]}}\one{b}\ot \nul{\two{[c]}}\ot \two{b},\\
    \epsilon_{H}([c]\ot b)=&\epsilon_{C}([c])\epsilon_{B}(b),\\
    S_{H}([c]\ot b)=& S_{C}(\nul{[c]})\ot S_{B}(\eins{[c]}b),\\
    s(b)=&\phi(\one{b})\otimes \two{b},\\
    t(b)=&1\otimes b,\\
    \Delta([c]\otimes b)=&\one{[c]}\otimes 1\otimes_{B}\two{[c]}\otimes b,\\
    \epsilon([c]\otimes b)=&\epsilon_{C}([c])b,\\
    S([c]\otimes b)=&S_{C}([c])\phi(\one{b})\otimes \two{b},\\
    \alpha([c]\otimes b)=&\beta(c)(\one{\one{b}}\otimes\two{\one{b}}\otimes \two{b})=c^{\hat{1}}\one{\one{b}}\otimes c^{\hat{2}}\two{\one{b}}\otimes c^{\hat{3}}\two{b}.
\end{align*}
\end{theorem}

\section{Finite dimensional coherent Hopf 2-algebras and examples}\label{sec. Finite dimensional coherent Hopf 2-algebra and examples}

Now we will give an explicit example of a coherent Hopf 2-algebra based on the Cayley algebra. By \cite{Majid09}, the unital basis of Cayley algebras $\mathcal{G}_{n}:=\{\pm e_{a}\,\,| \,\, a\in \mathbb{Z}_{2}^{n} \}$ is a quasigroup, with the product controlled by a 2-cochain $F: \mathbb{Z}_{2}^{n}\times \mathbb{Z}_{2}^{n}\to k^{*}$, more precisely, $e_{a}e_{b}:=F(a, b)e_{a+b}$. From now on, we define $e_{a}^{0}:=e_{a}$ and $e_{a}^{1}:=-e_{a}$, i.e. $\mathcal{G}_{n}=\{ e_{a}^{i}\,\,| \,\, a\in \mathbb{Z}_{2}^{n},\,\, i\in \mathbb{Z}_{2} \}$, so we have $e_{a}^{i}e_{b}^{j}=F(a, b)e_{a+b}^{i+j}$. We define $k\mathcal{G}_{n}$ as the linear extension of $\mathcal{G}_{n}$, which is a Hopf quasigroup with the coalgebra structure given by $\Delta(u)=u\otimes u$, $\epsilon(u)=1$, and $S(u):=u^{-1}$ on the basis elements.\\
As we already know from \cite{Majid09} that 
\begin{align*}
     k\mathcal{G}_{n}\simeq
  \begin{cases}
    \mathbb{C}       & \quad \text{if } n=1  \\
    \mathbb{H}       & \quad \text{if } n=2  \\
    \mathbb{O}       & \quad \text{if } n=3. 
  \end{cases}
\end{align*}
 We also have 
\begin{align*}
     N_{k\mathcal{G}_{n}}\simeq
  \begin{cases}
    \mathbb{C}       & \quad \text{if } n=1  \\
    \mathbb{H}       & \quad \text{if } n=2  \\
    \mathbb{R}       & \quad \text{if } n=3.  
  \end{cases}
\end{align*}
By [\cite{Majid09}, Prop 3.6], we know $\mathcal{G}_{n}$ is quasiassociative and $B:=k[\mathcal{G}_{n}]$ given by functions on $\mathcal{G}_{n}$ is a Hopf coquasigroup. Let $f_{a}^{i}\in k[\mathcal{G}_{n}]$ be the delta function on each element of $\mathcal{G}_{n}$, i.e. $f_{a}^{i}(e_{b}^{j})=\delta_{a, b}\delta_{i,j}$.  We can see $k[\mathcal{G}_{n}]$ is an algebra with generators $\{ f_{a}^{i}\,\,| \,\, a\in \mathbb{Z}_{2}^{n},\,\, i\in \mathbb{Z}_{2} \}$ subject to the relations: 
\begin{align*}
     f_{a}^{i}f_{a'}^{i'}=
  \begin{cases}
    f_{a}^{i}       & \quad \text{if } a=a' \quad\text{and}\quad i=i' \\
    0  & \quad \text{otherwise.} 
  \end{cases}
\end{align*}
The unit of $k[\mathcal{G}_{n}]$ is $\sum_{a\in \mathbb{Z}_{2}^{n},
i\in \mathbb{Z}_{2}} f_{a}^{i}$. The coproduct, counit and antipode are given by
\begin{align}
    \Delta_{B}(f_{a}^{i}):=&\sum_{\substack{b+c=a\\j+k=i}}F(b, c)f_{b}^{j}\otimes f_{c}^{k}.\\
    \epsilon_{B}(f_{a}^{i}):=&\delta_{a, 0}\delta_{i, 0}.\\
    S_{B}(f_{a}^{i}):=&F(a, a)f_{a}^{i}.
\end{align}
The previous structures make $k[\mathcal{G}_{n}]$ a Hopf coquasigroup. Now we will show $(k[\mathcal{G}_{0}], k[\mathcal{G}_{n}], \pi)$ is quasi coassociative, where $\pi:k[\mathcal{G}_{n}]\to k[\mathcal{G}_{0}]$ is the canonical projection map given by the pull back of the inclusion  $\{-e_{0}, e_{0}\}\subseteq \mathcal{G}_{n}$.
First, since $F(a, 0)=1$, we can see $(\pi\ot \id)\circ\Delta (f^{i}_{a})=\sum_{j+k=i}F(0,a)f^{j}_{0}\ot f^{k}_{a}=\sum_{j+k=i}f^{j}_{0}\ot f^{k}_{a}$, and $(k[\mathcal{G}_{0}], k[\mathcal{G}_{n}], \pi)$ is a coassociative pair. Second, we have 
\begin{align*}
    \one{\one{x}}S_{B}(\two{x})\otimes \two{\one{x}}&\in B\otimes I_{B}\\
    \one{x}S_{B}(\two{\two{x}})\otimes \one{\two{x}}&\in B\otimes I_{B},
\end{align*}
for any $x\in I_{B}$. Indeed, let $x\in I_{B}=\ker(\pi)$, then $x$ is a linear combination of $f_{a}^{i}$ with $a\neq 0$.
Without losing generality (since every map below is linear), assuming $x=f_{a}^{i}$, we can see

\begin{align}
    \one{x}S_{B}(\two{\two{x}})\ot \one{\two{x}}=\sum_{\substack{b+c+d=a\\j+k+l=i}}F(b, c+d)F(c, d)F(d, d) f_{b}^{j}f_{d}^{l}\ot f_{c}^{k}, 
\end{align}
the right hand side of the equality is not zero only if $b=d$. As a result $c$ is equal to $a$, and $\one{x}S_{B}(\two{\two{x}})\otimes \one{\two{x}}\in B\otimes I_{B}$. Similarly, we also have $ \one{\one{x}}S_{B}(\two{x})\otimes \two{\one{x}}\in B\otimes I_{B}$. Moreover, we can see the left adjoint coaction on $C$ is trivial, namely, $\Ad([f^{i}_{0}])=1\ot [f^{i}_{0}]$.
Finally, recall the comultiplicative coassociator (\ref{coassociator}) $\beta: B\to B\otimes B\otimes B$  
\begin{align*}
    \beta(b)=\one{\one{b}}\one{\one{S_{B}(\two{b})}}\otimes \one{\two{\one{b}}}\two{\one{S_{B}(\two{b})}}\otimes \two{\two{\one{b}}}\two{S_{B}(\two{b})}
\end{align*}
 for any $b\in B$. We can see $I_{B}\subseteq \ker(\beta)$. Indeed, without losing generality,
let $x=f_{a}^{i}$ with $a\neq 0$, we have:
\begin{align*}\label{equ. beta}
    \beta(x)=\beta(f_{a}^{i})=&\sum_{\substack{j+k+l+m+n+p=i\\b+c+d+e+f+g=a}}F(b+c+d, e+f+g)F(b, c+d)F(c, d)F(e, f+g)F(f, g)\\
    &F(e, e)F(f, f)F(g, g)f_{b}^{j}f_{g}^{p}\otimes f_{c}^{k}f_{f}^{n}\otimes f_{d}^{l}f_{e}^{m}.
\end{align*}
Since $a\neq 0$, we can see the right hand side of the above equation is zero (by using $b+c+d+e+f+g=a$). So $I_{B}$ belongs to the kernel of $\beta$. By Definition \ref{def. quasi coassociative},  we can see $(k[\mathcal{G}_{0}], k[\mathcal{G}_{n}], \pi)$ is quasi coassociative. Thus by Theorem \ref{thm. main theorem} there is a coherent Hopf 2-algebra structure, with $C=  k[\mathcal{G}_{0}]$, and $H= k[\mathcal{G}_{0}]\otimes k[\mathcal{G}_{n}]$.
More precisely, the Hopf algebroid structure of $H$ is given by:
\begin{align*}
    \Delta(f_{0}^{i}\otimes f_{a}^{l})&=\sum_{\substack{j+k=i}}f_{0}^{j}\otimes 1\otimes_{B} f_{0}^{k}\otimes f_{a}^{l};\\
    \epsilon(f_{0}^{i}\otimes f_{a}^{l})&=\delta_{i,0} f_{a}^{l};\\
    S(f_{0}^{i}\otimes f_{a}^{l})&=f_{0}^{i}\otimes f_{a}^{l-i};\\
    s(f_{a}^{l})&=\sum_{\substack{m+n=l}}f_{0}^{m}\otimes f_{a}^{n};\\
    t(f_{a}^{l})&=1\otimes f_{a}^{l}.
\end{align*}
Since the left adjoint coaction of $C$ is trivial, the Hopf coquasigroup structure of $H$ is given by:
\begin{align*}
    \blacktriangle(f_{0}^{i}\otimes f_{a}^{l})=&\sum_{\substack{m+n=l\\b+c=a\\j+k=i}}F(b, c)f_{0}^{j}\otimes f_{b}^{m}\otimes f_{0}^{k}\otimes f_{c}^{n};\\
    \epsilon_{H}(f_{0}^{i}\otimes f_{a}^{l})=&\epsilon_{B}(f_{0}^{i}f_{a}^{l})=\delta_{i, 0}\delta_{l, 0}\delta_{a, 0};\\
    S_{H}(f_{0}^{i}\otimes f_{a}^{l})=&F(a, a)f_{0}^{i}\otimes f_{a}^{l}.
\end{align*}
Recall $\alpha: H\to B\otimes B\otimes B$ in  (\ref{def. coassociator}), we have
\begin{align*}
  \alpha(f_{0}^{i}\otimes f_{a}^{l})=\sum_{\substack{k+m+n=l\\b+c+d=a}}\beta(f_{0}^{i})(F(b+c, d)F(b, c)f_{b}^{k}\otimes f_{c}^{m}\otimes f_{d}^{n}).  
\end{align*}
From the formula of $\beta$, we can see that $\alpha$ is controlled by a 3-coboundary $\partial F$. In fact,
\begin{align*}
    \beta(f^{i}_{0})=&\sum_{\substack{j, k, l\in \mathbb{Z}^{2}\\b, c, d\in \mathbb{Z}^{2}_{n}}}F(b+c+d, b+c+d)F(b, c+d)F(c, d)F(d, c+b)F(c, b)\\
    &F(d, d)F(c, c)F(b, b)f_{b}^{j}\otimes f_{c}^{k}\otimes f_{d}^{l}\\
    =& \sum_{\substack{j, k, l\in \mathbb{Z}^{2}\\b, c, d\in \mathbb{Z}^{2}_{n}}}F(b+c+d, b+c+d)\partial F(b, c, d)F(d, d)F(c, c)F(b, b)f_{b}^{j}\otimes f_{c}^{k}\otimes f_{d}^{l},
\end{align*}
where $\partial F$ is the 3-coboundary given by \cite{Majid09}:
\begin{align*}
    \partial F(b, c, d)=\frac{F(b, c+d)F(c, d)}{F(d, c+b)F(c, b)}=F(b, c+d)F(c, d)F(d, c+b)F(c, b).
\end{align*}

\vspace{.5cm}
	
\noindent
{\bf Acknowledgment:}
I thank Prof. Dr. Giovanni Landi, Prof. Dr. Ralf Meyer, Prof. Dr. Chenchang Zhu and  Daan Van De Weem for many useful discussions, and also Dr. Song Cheng, Adam Magee and Veronica Fantini for their proof reading.

\end{document}